\documentclass{jams-l}
\usepackage{color}
\usepackage{hyperref}
\newtheorem{theorem}{Theorem}[section]
\newtheorem{lemma}[theorem]{Lemma}

\theoremstyle{definition}
\newtheorem{definition}[theorem]{Definition}

 \newtheorem{proposition}[theorem]{Proposition}
\newtheorem{corollary}[theorem]{Corollary}
\theoremstyle{remark}
\newtheorem{remark}[theorem]{Remark}

\newcommand{\R}{\mathbb{R}}

\numberwithin{equation}{section}

\usepackage{mathrsfs}

\begin{document}

\title[Blowup solutions for the shadow Gierer-Meinhardt system]{Diffusion-induced blowup solutions for the shadow limit model of a singular Gierer-Meinhardt  system}

 \author{G. K. DUONG}
\address{Department of Mathematics, New York University in Abu Dhabi, Saadiyat Insland. P. O. Box 129188, Abu Dhabi, United Arab Emirates }

  \author{N. I. Kavallaris }
  \address{Department of Mathematical and Physical Sciences,
Faculty of Science  and Engineering, University of Chester, Thornton Science Park, Pool Lane, Ince, CH2 4NU, Chester, UK  }
\author{H.  Zaag}
\address{LAGA,  Sorbonne Paris Nord University,  CNRS (UMR 7539), F- 93430, Villetaneuse, France}



\bigskip
\begin{center}\thanks{\today}\end{center}


\maketitle
 
 \textit{Abstract:} 
 In the current paper,  we  provide a thorough investigation of the blowing up behaviour induced via diffusion of the solution of the following  non local problem
 \begin{equation*}
 \left\{\begin{array}{rcl}
   \partial_t u    &=&  \Delta u  -  u + \displaystyle{\frac{u^p}{ \left(\mathop{\,\rlap{-}\!\!\int}\nolimits_\Omega    u^r dx  \right)^\gamma }}\quad\text{in}\quad   \Omega \times  (0,T), \\[0.2cm]
\frac{ \partial u}{ \partial \nu}  & = &  0   \text{ on  }    \Gamma =  \partial \Omega \times  (0,T),\\
u(0) & = & u_0,
\end{array} \right.
\end{equation*}
where $\Omega$ is a bounded domain in $\R^N$ with smooth boundary $\partial \Omega;$ such problem is derived  as the shadow limit of a singular Gierer-Meinhardt system, cf. \cite{KSN17, NKMI2018}.  Under the Turing type condition
$$    \frac{r}{p-1} < \frac{N}{2},  \gamma r \ne p-1, p >1  $$
we construct a solution   which blows up in finite time and  only  at an interior point $x_0$  of  $\Omega,$ i.e.  
$$ u(x_0, t) \sim   (\theta^*)^{-\frac{1}{p-1}} \left[\kappa (T-t)^{-\frac{1}{p-1}} \right], $$
where 
$$ \theta^* := \lim_{t \to T} \left(\mathop{\,\rlap{-}\!\!\int}\nolimits_\Omega    u^r dx \right)^{- \gamma} \text{ and } \kappa = (p-1)^{-\frac{1}{p-1}}.    $$
More precisely, we also give a description on the  final asymptotic profile  at the blowup point
$$   u(x,T)  \sim   ( \theta^* )^{-\frac{1}{p-1}}  \left[  \frac{(p-1)^2}{8p} \frac{|x-x_0|^2}{ |\ln|x-x_0||}  \right]^{ -\frac{1}{p-1}}   \text{ as }  x  \to 0, $$
and thus we unveil the form of  the Turing patterns occurring in that case due to driven-diffusion instability.

The applied technique for the  construction of the preceding blowing up solution  mainly relies  on the approach developed in   \cite{MZnon97} and  \cite{DZM3AS19}.       
  
\newpage
\tableofcontents

\newpage

\section{Introduction}\label{section-intro}
In as early as 1952, A. Turing in his seminal paper \cite{TPTRS1952}
attempted, by using reaction-diffusion systems, to model the phenomenon of morphogenesis,
the regeneration of tissue structures in hydra, an animal of a few millimeters in length
made up of approximately 100,000 cells. Further observations on the morphogenesis in
hydra led to the assumption of the existence of two chemical substances (morphogens), a
slowly diffusing (short-range) activator and a rapidly diffusing (long-range) inhibitor. A.
Turing, in \cite{TPTRS1952}, indicates that although diffusion has a smoothing and trivializing effect
on a single chemical, for the case of the interaction of two or more chemicals different
diffusion rates could force the uniform steady states of the corresponding reaction-diffusion
systems to become unstable and to lead to non-homogeneous distributions of such
reactants. Since then, such a phenomenon is known as Turing-type instability or
diffusion-driven instability (DDI).
Scrutinizing  Turing’s idea further, Gierer and Meinhardt \cite{GMK1972}, proposed in 1972 the following
activator–inhibitor system to model the
regeneration phenomenon of hydra located in a domain $\Omega\subset \R^N, N\geq 1$  
\begin{equation}\label{Gierer-Meinhardt}
\left\{   \begin{array}{rcl}
\partial_t u    &=& 	  \epsilon^2 \Delta u   - u  + \frac{u^p}{v^q}  \text{ in }  \Omega  \times  (0,T),\\[0.2cm]
\tau \partial_t v   &=&  D \Delta  v    - v   + \frac{u^r}{v^s} \text{ in }  \Omega  \times  (0,T),\\[0.2cm]
\frac{ \partial u}{ \partial \nu }  & =&   \frac{ \partial v}{ \partial \nu } =0   \text{ in }  \partial \Omega  \times  (0,T),\\[0.2cm]
u( 0 ) &=&  u_0 > 0   \text{ and  }  v(0)  = v_0 > 0.
\end{array} 
 \right.
\end{equation}
Here $\nu$  stands for the  unit outer  normal  vector  to $\partial 
\Omega,$ whilst $u$ and  $v$ are the concentrations of  the    activator  and the   inhibitor, respectively. Besides,   $\epsilon$ and $ D$ represent the  diffusing  coefficients  and exponents $p,q,r,s$ measuring the morphogens interactions satisfy 
$$p>1, q,r   > 0   \text{ and }  s > -1.$$ 
A biologically interesting case arises when the activator diffuses much faster compared to the inhibitor. So in the case where $D \to +\infty,$ dividing the second one in \eqref{Gierer-Meinhardt} by  $D$,   we take formally that  for any $ t \in (0,T)$ and  thanks to  the  Neumann boundary condition, activator's concentration  $v$   will   be  spatial homogeneous, i.e. $ v(x,t)   =   \xi (t),$ (this system was first introduced by Keener \cite{KSAM78}; see also \cite{KSN17}, \cite{NKMI2018} and \cite{KBM20}; a rigorous proof for a version of Gierer-Meinhardt system can be found in  \cite{MCMVJM17} and \cite{MCHKSN18}, whilst for the case of general reaction-diffusion systems the interested reader can check \cite{BK19}).  Next, integrating   the second equation   in     \eqref{Gierer-Meinhardt},  we  finally derive  the shadow  system for  $u$ and $\xi$
\begin{equation}\label{Shadow-model-xi-homogeneous}
\left\{   \begin{array}{rcl}
\partial_t u    &=& 	  \epsilon^2 \Delta u   - u  + \frac{u^p}{v^q}  \text{ in }  \Omega  \times  (0,T),\\[0.2cm]
\tau   \partial_t \xi   &=&     - \xi      + \frac{ \def\avint{\mathop{\,\rlap{-}\!\!\int}\nolimits} \avint_\Omega
u^r\,dx}{ \xi^s} \text{ in }  \Omega  \times  (0,T),\\[0.2cm]
\frac{ \partial u}{ \partial \nu }  &=& 0   \text{ in }  \partial \Omega  \times  (0,T),\\[0.2cm]
u( 0 ) &=&  u_0 > 0, 
\end{array} 
 \right.
\end{equation}
where 
$$\def\avint{\mathop{\,\rlap{--}\!\!\int}\nolimits} \avint_\Omega
u^r\,dx =   \frac{1}{ |\Omega|}   \int_{\Omega}  u^r dx.   $$
We  now focus on the case $ \tau  = 0,$ that is when inhibitor's response  rate is quite small  against inhibitor's growth, and thus by the second  equation in \eqref{Shadow-model-xi-homogeneous} we derive 
\begin{equation}\label{solution-xi-t}
\xi(t)   =     \left(  \def\avint{\mathop{\,\rlap{--}\!\!\int}\nolimits} \avint_\Omega
u^r\,dx  \right)^\frac{1}{s+1}.
\end{equation}     
Next plugging \eqref{solution-xi-t} into  \eqref{Shadow-model-xi-homogeneous}, we  finally obtain the following non-local  system
\begin{equation}\label{equa-u-non-local}
\left\{\begin{array}{rcl}
   \partial_t u    &=&  \Delta u  -  u + {\displaystyle\frac{u^p}{ \left(\mathop{\,\rlap{-}\!\!\int}\nolimits_\Omega    u^r dx  \right)^\gamma }}\quad\text{in}\quad   \Omega \times  (0,T),  \\
\frac{ \partial u}{ \partial \nu}    &=&  0   \text{ on  }      \partial \Omega \times  (0,T)\\
u(0) & = & u_0 \geq 0,
\end{array} \right.
\end{equation}
where $  \gamma  =  \frac{1}{s+1}$  and $\epsilon = 1$ for simplicity.  

Notably under the Turing condition
\begin{equation}\label{Turing-cond}
p-r\gamma<1
\end{equation}
the spatial homogeneous solution of \eqref{equa-u-non-local} given by
\begin{equation*}
\frac{du}{dt}=-u+u^{p-r\gamma},\quad u(0)=u_0\geq 0,
\end{equation*}
never exhibits blow-up, since the non-linearity is sublinear, and its unique stationary state
$u = 1$ is asymptotically stable, cf. \cite{KSN17}. On the other hand, under condition \eqref{Turing-cond} the solution of \eqref{equa-u-non-local} when domain $\Omega$ is the unit $N-$dimensional sphere, $N\geq 3,$  exhibits a Type I finite-time blow-up only at the center of the sphere, cf. \cite[Theorem 3.7]{KSN17}, and thus a Turing type instability (in the form of diffusion-driven blow-up) occurs. Notably, the blowup time $T$ can be done relatively small  by reducing properly the size of the initial data, cf. \cite{KSN17} and \cite{NKMI2018}. For analogous blowup results for the Gierer-Meinhardt system on an evolving domain and for a non-local Fisher-KPP equation one can see \cite{KBM20} and \cite{KL20} respectively. 

The main purpose of the current work is to describe the form of the developing Turing instability (blowup) patterns for the solution of problem  \eqref{equa-u-non-local} in a region of any blowup point.
To this end we first note that the non-local equation in \eqref{equa-u-non-local} is closely associated to the standard heat equation given in \eqref{classical-heat-equation-intro}.   Indeed,  if  we ignore  the linear term $-u$ and  take $\gamma =0$, then it turns out the classical nonlinear heat equation
\begin{equation}\label{classical-heat-equation-intro}
  \partial_t u = \Delta u + u^p.  
\end{equation}
Let us now recall some well known results link with the blowup behaviour of \eqref{classical-heat-equation-intro}. 
Firstly, \eqref{classical-heat-equation-intro}  is well-posed in $L^\infty(\Omega),$ hence, for any $u_0 \in L^\infty (\Omega) $, either the solution is global or it blows up in finite time $T=T(u_0),$ i.e
\begin{equation}\label{defi-norm-L-infty}
    \|  u (\cdot, t)\|_{L^\infty(\Omega)}    \to  + \infty  \text{ as }   t \to T.
\end{equation}

There is also a great amount of  works related  to  the blowing up behaviour of \eqref{classical-heat-equation-intro} and its various generalizations. In particular, the construction of blowup solutions of \eqref{classical-heat-equation-intro} is discussed in \cite{Brejde92}, \cite{BKnon94}, \cite{MZdm97}, \cite{MZnon97} whilst one can find in \cite{TZpre15}, \cite{DNZtunisian-2017}, \cite{NZsns16} analogous constructions for perturbed nonlinear source terms. Besides, in \cite{MZjfa08} \cite{NZ2017}, \cite{DNZMAMS20} such blowup solutions are build for complex Ginzburg-Landau equations and in  \cite{NZcpde15},  \cite{DJFA2019}, \cite{DJDE2019}  for   complex-valued heat equations which lack  variational structure; the case of parabolic systems is examined in \cite{GNZsysparabolic2016}. Additionally, in  \cite{DZM3AS19} a singular solution associated with a duality concept to blowup phenomenon, called quenching (or touch-down in MEMS literature), is constructed.  In particular, the   authors in \cite{DZM3AS19} developed further the idea of \cite{MZnon97} to describe the quenching behaviour of a non-local problem arising from MEMS industry (see more in  \cite{DKN20}, \cite{GKDCDS2012}, \cite{GSSAIM2015}, \cite{NIKN16}  and the references therein). 

\medskip
\noindent
\textbf{Acknowledgement:}
G. K.  Duong would like to thank Professor for his constant support and valuable comments during his appointment at New York University in Abu Dhabi, UAE.       

\section{Main results}
In the current paper our main aim is to  construct  a solution to  \eqref{equa-u-non-local}   which blows up in finite time $0<T< 1$ and  a point $x_0 \in  \Omega,$ that is
$$ u(x_0, t) \to   +\infty, \text{ as } t \to T. $$
In that case \eqref{defi-norm-L-infty} holds, and  our result holds true under the Turing condition, i.e. when 
\begin{equation}\label{subcritical-condition}
    \frac{r}{p-1}  < \frac{N}{2}\quad \mbox{and}\quad  \gamma r \neq p-1, p>1.
\end{equation}
 \begin{theorem}\label{theorem-existence}    Let  $\Omega$ be a smooth and bounded domain in $ \mathbb{R}^N$ and also assume  \eqref{subcritical-condition}  is satisfied.  Then, for any arbitrary   $x_0  \in \Omega$ there exist    initial data  $u_0   \geq  0    $  such that the solution of   \eqref{equa-u-non-local}    blows up in finite time $T(u_0)<1$,  only  at $x_0$. Moreover, we have the following     blowup profiles
 \begin{itemize}
 \item[$(i)$] The   intermediate profile   for  all $ t \in (0,T)$
 \begin{eqnarray}\label{estima-T-t-u-theorem-varphi-0}
 & & \left\|    (T-t)^{\frac{1}{p-1}} u(\cdot  , t  )  - ( \theta^*)^{-\frac{1}{p-1}}    \varphi_0 \left(   \frac{x-x_0}{ \sqrt{(T-t) |\ln(T-t)|}}  \right) \right\|_{L^\infty(\Omega)} \label{intermediate-profile}\\
  &&   \hspace{7cm}  \leq  \frac{C}{1 + \sqrt{|\ln(T-t)|}} \nonumber,
 \end{eqnarray}
where 
\begin{eqnarray}
\theta^*  =   \lim_{t \to T }       \left(  \def\avint{\mathop{\,\rlap{--}\!\!\int}\nolimits} \avint_\Omega
u^r\,dx  \right)^{-\gamma}   > 0,
\end{eqnarray}
and 
\begin{equation}\label{defi-varphi-0}
\varphi_0 \left(  z \right)    =  \left( p -1   +    \frac{(p-1)^2}{4p} |z|^2    \right)^{ -\frac{1}{p-1}}.
\end{equation}
\item[$(ii)$] The final profile at $t =T$ is provided by $ u(x,t)   \to    u^* (x) \in C^2( \Omega  \setminus  \{ x_0\})$ uniformly     on  compact sets   of  $\Omega  \setminus  \{ x_0\}$.  In particular, near $x_0$,  solution $u^*$ behaves  as
$$   u^*(x)      \sim   (\theta^*)^{ -\frac{1}{p-1}} \left[ \frac{(p-1)^2}{8p}  \frac{|x- x_0|^2}{|\ln|x -x_0||}\right]^{ -\frac{1}{p-1}}  \text{ as }  x \to x_0.   $$
 \end{itemize}
 \end{theorem}
 Consequently,   we also describe the asymptotic of $\|u\|^k_{L^k(\Omega)}$ in the following.
 
 \begin{corollary}[Behavior of $L^k$ norm at blowup time]\label{corollary-u-L-k} Let  $ u$ be the constructed blowup solution given by Theorem \ref{theorem-existence}, and $k >0$. Then, the following hold:
 \begin{itemize}
     \item[$(i)$] If  $ \frac{k}{p-1} < \frac{N}{2} $, then 
     $$ \| u\|_{L^k(\Omega)}^k \to C(k) < +\infty, \text{ as } t \to T.$$
     \item[$(ii)$] If $ \frac{k}{p-1} > \frac{N}{2} $, then 
     \begin{equation}\label{equivlence-norm-U-L-k}
     \|u(t)\|_{L^k(\Omega)}^k  = \left( (\theta^*)^{-\frac{k}{p-1}} \int_0^\infty \varphi_0^k(r) r^{N-1} dr + o_{t \to T}(1)\right) (T-t)^{\frac{N}{2} - \frac{k}{p-1}} |\ln(T-t)|^{\frac{N}{2}}. 
 \end{equation}
 \item[$(iii)$] If $ \frac{k}{p-1}  - \frac{N}{2} =0$, then 
 $$  \|u(t)\|_{L^k(\Omega)}^k  =  \left( (\theta^*)^{-\frac{k}{p-1}} \frac{k}{\frac{N}{2} +1}\int_0^\infty \varphi_0^{p-1+k}(r) r^{N-1} dr + o_{t \to T}(1)\right) |\ln(T-t)|^{\frac{N}{2} +1}.  $$
 \end{itemize}
 \end{corollary}
 
 %
 
 \begin{remark}[Open problem] Note that in  \eqref{subcritical-condition}, we impose 
 $$ \gamma r  \neq p-1,  $$
since otherwise, due to \eqref{defi-theta-by-U},  our method does not  work.  
 \end{remark}
 \begin{remark} In Theorem \ref{theorem-existence} a solution blowing up at an arbitrary $x_0$ is constructed. However, by the translation $x -x_0$ we can always derive a blowing  up solution at $0.$ So, we need to prove Theorem \ref{theorem-existence} only for $x_0 =0.$ In addition, we can apply  the technique  of \cite{Mercpam92}, and we can establish a blowup solution  at only $k$ points $x_1,...,x_k$ with blowup profiles provided by Theorem \ref{theorem-existence} by replacing $L^{\infty}(\Omega)$ with $L^{\infty} (|x-x_j| \leq \epsilon_j)$ at each blowup point $x_j.$
\end{remark}
 \begin{remark}[Stability of the blowup profile] Let us consider $\hat u$, the constructed solution in Theorem \ref{theorem-existence} with initial data $\hat u_0$ which blows up at time $\hat T$ and at the point $\hat x_0.$  Then there exists an open neighborhood  of $\hat u_0$  in a sub-space of    $C(\bar \Omega) $ and with a suitable topology, named $\hat{\mathcal{U}}_0$ such that for all  $u_0 \in  \hat{\mathcal{U}}_0$, the corresponding solution of \eqref{equa-u-non-local}, blows up at $x(u_0)$ and only at $T(u_0)$ with blowup profiles given by in Theorem \ref{theorem-existence} by replacing $x_0$ to $x(u_0).$ In particular, we have
 $$  (x(u_0), T(u_0), \theta^*(u_0))  \to (\hat x_0, \hat T, \hat{\theta}^*) \text{ as } u_0 \to \hat u_0, $$
 where 
 $$ \theta^*(u_0) =\lim_{t \to T(u_0) }       \left(  \def\avint{\mathop{\,\rlap{--}\!\!\int}\nolimits} \avint_\Omega
u^r\,dx  \right)^{-\gamma}\quad\mbox{and}\quad \hat{\theta}^*=\lim_{t \to T(\hat u_0) }       \left(  \def\avint{\mathop{\,\rlap{--}\!\!\int}\nolimits} \avint_\Omega
\hat{u}^r\,dx  \right)^{-\gamma}.$$
The stability result follows by the interpretation of the parameters of the finite-dimensional problem in terms of the blowup time and the blowup point, see more in  \cite{MZdm97}.
 \end{remark}
 \section{Formal approach}
In this  section, we aim at giving a formal approach which explains how  the profile  in Theorem \ref{theorem-existence}  is derived.  Firstly, let us denote    
\begin{equation}\label{defi-theta-}
\theta (t)   =    \frac{1}{   \left(\mathop{\,\rlap{-}\!\!\int}\nolimits_\Omega    u^r(t) dx \right)^\gamma }. 
\end{equation} 
Henceforth,  we  rewrite   equation  \eqref{equa-u-non-local}  as follows
\begin{equation}\label{equa-main-short}
\partial_t u  = \Delta u - u     +  \theta(t) u^p. 
\end{equation}
We can see that  $\theta(t) $ strongly affects the blowup dynamic of $u$.    Let us assume  that $u$  blows  up in finite time $T$  and   at  the origin   $0 \in \Omega$ (without loss of generality).   
As we are proceeding formally, we believe the three following cases are important for our analysis:
  \begin{eqnarray}
  \theta(t)   & \to &  0     \text{ as }   t \to T   \label{case-theta-to-0},\\
  \theta(t)   & \to  &  \theta^* >0   \text{ as }  t \to T  \label{case-theta-to-thata-*},\\
  \theta  (t)  & \to &  +\infty   \text{ as }   t \to T  \label{case-theta-to-infty}.
  \end{eqnarray}
 In particular, \eqref{case-theta-to-infty}  is excluded   by   Theorem 3.1 and Remark 3.2  given in  \cite{KSN17}.  In the context of this work,  we  aim to handle only the convergent case i.e  \eqref{case-theta-to-thata-*}, whilst the case \eqref{case-theta-to-0} will be treated in a forthcoming paper.  We note  that the convergent  case \eqref{case-theta-to-thata-*} was already encountered in  \cite{DZM3AS19} where the authors managed to control the non-local term in a MEMS model, up to the quenching time. 
 
 We rewrite  \eqref{equa-main-short}  as  follows
 \begin{equation*}
 \partial_t  u  =  \Delta  u   +    \theta^*   u^p  +    (\theta(t) -\theta^*) u^p - u
 \end{equation*}
 and we formally neglect  the term 
 $$ (\theta(t) -\theta^*) u^p - u,$$
 since,  it  is   relatively small compared to the main non-linear  term 
 $$  \theta^* u^p.$$        
 Therefore,   it is important to study the limit problem
 \begin{eqnarray*}
\partial_t u  =   \Delta u  +     \theta^*  u^p,
\end{eqnarray*}
instead of the full model \eqref{equa-main-short}.  In addition to that, using the following re-scaled form
$$   u(x,t) = ( \theta^*)^{-\frac{1}{p-1}}  U(x,t ),$$
$U$ then  solves
\begin{equation}\label{classicle-heat-equa}
\partial_t U    = \Delta U   + U^p,
\end{equation}
i.e.  the standard heat equation whose blowup solutions have been studied thoroughly, cf. Section \ref{section-intro}. In particular,  the approach developed in  \cite{MZnon97} to construct a very precised  asymptotic profile of  blowup solutions to \eqref{classicle-heat-equa} is quite related to our work. Indeed, for some positive constants  $K_0$ and $ \epsilon_0$,  we  now cover  $\Omega$  by 
$$  \Omega    =   P_1 (t)    \cup P_2(t)  \cup  P_3(t) , $$
where 
\begin{eqnarray}
P_1 (t)  &=&  \left\{   x \in \mathbb{R}^N  \left| \right.  |x| \leq K_0 \sqrt{(T-t) |\ln(T-t)|}    \right\} \label{defini-P-1-t},\\
P_2 (t) & =& \left\{   x \in \mathbb{R}^N   \left| \right.   \frac{K_0}{4} \sqrt{(T-t) |\ln(T-t)|}  \leq  |x| \leq \epsilon_0   \right\} \label{defini-P-2-t},\\
P_3(t) & =&  \left\{  x \in \mathbb{R}^N  \left|  \right.  |x| \geq \frac{\epsilon_0}{4}       \right\}   \cap  \Omega  .\label{defini-P-3-t}
\end{eqnarray}
Here and throughout the proof, we assume that $T<1$. Indeed, all the key estimates of our proof (Propositions 4.4, 4.5, etc.) hold for $T$ small enough.
On each  domain, we  will control  the solution with a suitable behavior.    In particular, we also name them by  $P_1(t)$-\textit{blowup region};  $P_2(t)$-\textit{intermediate region}; and   $P_3(t)$- \textit{regular region}.
\begin{center}
\textbf{ Asymptotic blowup profile in region $P_1$:  }
\end{center}
The region $P_1$ is the region where the blowup phenomenon  mainly occurs. Besides,  the blowup dynamic is described via   the following similarity variable introduced  in \cite{GKcpam85}
\begin{equation}\label{similarity-variable}
y  = \frac{x}{\sqrt{T-t}},     \quad  s  = -\ln(T-t)  \text{ and  } W (y,s)  = (T-t)^{\frac{1}{p-1}} U(x,t). 
\end{equation}
Hence, $W$ solves
\begin{equation}\label{equa-W}
\partial_s W =  \Delta W - \frac{1}{2} y \cdot \nabla W - \frac{W}{p-1}     + W^p,   \forall   (y,s)  \in \Omega_s \times  [-\ln(T-t), +\infty),
\end{equation}
where   $\Omega_s   =  e^{\frac{s}{2}} \Omega $.  Following   \cite[page 149]{MZdm97},   the generic profile inside this region is given as following
$$ W(y,s) \sim  \varphi(y,s) = \left(p-1+\frac{(p-1)^2}{4p} \frac{|y|^2}{s}  \right)^{-\frac{1}{p-1}} + \frac{ \kappa N }{ 2p  s}. $$
\begin{center}
    \textbf{ Asymptotic of the intermediate profile in region $P_2$:  }
\end{center}
In region $P_2$ we try to control  a  re-scaled  function   $\mathcal{U}$ instead of $U$. For all $|x|$ small, we can define $t(x)$ as the unique solution  of 
\begin{eqnarray}
|x| &=& \frac{K_0}{4} \sqrt{ (T-t(x) ) | \ln(T-t(x))|} \text{ with  } t(x) < T.\label{c4defini-t(x)-}
\end{eqnarray}
Then,   we introduce the re-scaled $\mathcal{U}$ by
\begin{equation}\label{c4rescaled-function-U}
\mathcal{U} (x, \xi, \tau) =  \left(T-t (x) \right)^{\frac{1}{p-1}}  U \left( x + \xi \sqrt{ T- t(x)},   (T-t(x)) \tau  +  t(x)     \right),
\end{equation}
where $  \xi \in   (T - t(x))^{-\frac{1}{2}} ( \Omega - x)   $ 	and  $\tau \in \left[- \frac{t(x)}{T-t(x)}, 1 \right).$  Note that,  $t(x)$  is well  defined as long as  $\epsilon_0 $ is  small enough and  we have the following  asymptotic behaviour
$$  t(x)  \to  T, \text{ as }  x \to 0.$$
 For convenience, we introduce 
\begin{equation}\label{c4defini-theta}
\varrho (x) = T - t(x),
\end{equation}
so, it follows
$$\varrho(x) \to 0  \text{ as } x  \to 0,$$
and by virtue of \eqref{classicle-heat-equa},  we derive that $\mathcal{U}$ solves
 \begin{equation*}\label{c4equa-xi}
 \partial_\tau  \mathcal{U}  =   \Delta_\xi \mathcal{U}   +\mathcal{U}^p .
 \end{equation*}
Now we recall the main argument in \cite{MZnon97},  which demonstrates that is only sufficient to study  the dynamic of $\mathcal{U}$  on a small region of the local space $(\xi,\tau)$ defined by
 $$  |\xi| \leq  \alpha_0  \sqrt{|\ln( \varrho(x))|}  \text{ and }   \tau \in \left[ -\frac{  t(x)}{ \varrho(x)}, 1 \right).  $$
 When $\tau = 0$, we are in region $P_1(t(x))$; in fact in that case  $P_1 (t(x))$ and $P_2(t(x))$ have some overlapping by their definitions. Due to the imposed  constraints in region $P_1(t(x))$, we derive that $\mathcal{U}(x,\xi, 0)$ is flat in the sense that 
$$  \mathcal{U}(x, \xi, 0) \sim \left( p-1 + \frac{(p-1)^2}{4 p} \frac{K_0^2}{16}\right)^{-\frac{1}{p-1}}. $$ 
The main  idea is to show that this flatness is preserved for all $\tau \in [0, 1)$ (that is for all $t \in [t(x),T)$), in the sense that the solution does not depend substantially on space. For that purpose   $\mathcal{U}$ is regarded   as a perturbation  of  $\hat{ \mathcal{U}}(\tau),$ where  $\hat{ \mathcal{U}}(\tau)$ is defined as follows
 \begin{equation*}\label{c4equa-hat-mathcal-U}
 \left\{  \begin{array}{rcl}
 \partial_\tau  \hat{\mathcal{U}} (\tau)   & = &  \hat{ \mathcal{U}}^p (\tau),\\
 \hat{\mathcal{U}} (0)  & = &  \left(  p-1  + \displaystyle \frac{(p-1)^2}{4 p} \displaystyle \frac{K_0^2}{16} \right)^{-\frac{1}{p-1}},
\end{array} \right.
 \end{equation*}
and is explicitly given by 
 \begin{equation}\label{c4defin-hat-mathcal-U-tau}
 \hat{\mathcal{U}} (\tau) =  \left(   (p-1) (1 -\tau)  +  \displaystyle \frac{(p-1)^2}{4 p} \frac{K_0^2}{16}  \right)^{-\frac{1}{p-1}}.
 \end{equation}

\bigskip

\begin{center}
    \textbf{ Asymptotic profile in the regular region $P_3$:  }
\end{center}
 Using  the well-posedness of  the Cauchy problem  for equation  \eqref{classicle-heat-equa}, we derive the asymptotic profile of  the solution $U$ within that region as a  perturbation of initial data $U(0)$.

 \section{Formulation of the full problem}
 In the current section, we aim at  stating  the rigorous steps towards the proof of Theorem   \ref{theorem-existence}.
 
 \subsection{Similarity variable } Let  $ u$ be a solution  of  \eqref{equa-u-non-local}  then we introduce \begin{equation}\label{defi-U-by-u}
     U(x,t) = \theta(t)^{\frac{1}{p-1}} u(x,t),
 \end{equation}
 which by virtue of \eqref{defi-theta-} entails
 \begin{eqnarray}\label{defi-theta-by-U}
\theta (t)  =     \left(  \def\avint{\mathop{\,\rlap{--}\!\!\int}\nolimits} \avint_\Omega
U^r\,dx  \right)^{ -\frac{\gamma}{1- \frac{r \gamma}{p-1}}}.
 \end{eqnarray}
Next using equation \eqref{equa-u-non-local},  $U$ reads 
\begin{equation}\label{equa-U-theta-'-theta}
\partial_t  U     =  \Delta U          + U^p    +\left( \frac{1}{p-1} \frac{\theta'(t)}{\theta(t)}-1 \right) U,
\end{equation}
where $\theta(t)$ is defined as in \eqref{defi-theta-by-U}.  Now, we use the similarity variable introduced in \eqref{similarity-variable} to derive 
\begin{eqnarray}
\partial_s W =  \Delta W - \frac{1}{2} y \cdot \nabla W - \frac{W}{p-1}     + W^p + \left( \frac{1}{p-1}\frac{\bar \theta_s(s)}{\bar \theta(s)}  - e^{-s} \right)  W ,\label{equa-W-rerogous}
\end{eqnarray}
where 
\begin{equation}\label{defi-bar-theta-s}
    \bar \theta(s) = \theta(t(s)),\quad s = -\ln (T-t),
\end{equation}
and $y \in \Omega_s = e^{\frac{s}{2}} \Omega$. 

Note that there are some  technical difficulties arising by the evolution of $\Omega_s$ which we can overcome by using the approach introduced in \cite{MNZNon2016} (also used in \cite{DZM3AS19}), and resolves this techical issue via the extension of problem on $\mathbb{R}^N$. Indeed, let us introduce 
$\chi_0 \in C_0^\infty ([0,+\infty)) $, satisfying 
\begin{equation}\label{c4defini-chi-0}
supp (\chi_0) \subset [0,2], \quad 0 \leq \chi_0(x) \leq 1, \forall x \text{ and  }   \chi_0 (x)= 1, \forall x \in [0,1].
 \end{equation}
Then,   we   define  the following  function
\begin{equation}\label{c4defini-psi-M-0-cut}
  \psi_{M_0} (y,s)  =   \chi_0 \left(   M_0  y e^{- \frac{s}{2}}       \right), \text{ for some } M_0 > 0.
\end{equation}
Let us  introduce
\begin{equation}\label{c4defini-w-small}
w (y,s)  = \left\{   \begin{array}{rcl}
W (y,s)  \psi_{M_0} (y,s)  &  \text{  if }  &  y \in  \Omega_s,\\[0.3cm]
0  & & \text{otherwise}.  
\end{array}  \right.
\end{equation}
Using equation \eqref{equa-W-rerogous}, $w$  reads
\begin{eqnarray*}
\partial_s w = \Delta w  -  \frac{1}{2} y \cdot \nabla w - \frac{1}{p-1} w   +  w^p + \left( \frac{1}{p-1} \frac{ \bar\theta' (s)}{\bar \theta (s)} -e^{-s} \right) w + F(w,W),
\end{eqnarray*}
where $F(w,W)$ is given by
\begin{equation}\label{c4defini-F-1}
F(w,W) = \left\{   \begin{array}{rcl} & &
W \left[ \partial_s \psi_{M_0}   -  \Delta \psi_{M_0}  +  \frac{1}{2} y \cdot \nabla \psi_{M_0} \right] - 2  \nabla \psi_{M_0} \cdot \nabla W\\[0.3cm]
 & & +  \psi_{M_0}\left(W   \right)^{p} -w^{p}  ,     \text{ if } y\in \Omega e^{\frac{s}{2}},\\[0.4cm]
 & &  0,   \text{ otherwise}. 
\end{array}   \right. 
\end{equation}
Note that the nonlinear term $F(w,W)$ is quite 
similar to the term that
occurs in \cite{DZM3AS19} and thus can be neglected. More precisely,  the growth of the following term
$$ \left( \frac{1}{p-1} \frac{ \bar\theta' (s)}{\bar \theta (s)} -e^{-s} \right) w + F(w,W), $$
can be controlled, as it actually decays exponentially.

Then,  using  \cite{BKnon94} and  \cite{MZdm97}, we get the following blowup profile:
\begin{equation}\label{defi-varphi}
   \varphi(y,s) := \left( p-1 + \frac{(p-1)^2}{4 p} \frac{|y|^2}{s} \right)^{-\frac{1}{p-1}} + \frac{\kappa N }{2p s}, \quad \kappa= (p-1)^{-\frac{1}{p-1}}.
\end{equation}
We  now linearize around $\varphi$ 
\begin{equation}\label{defi-q=w-varphi}
    q =w - \varphi,
\end{equation}
hence, $q$ solves 
 \begin{equation}\label{c4equa-Q}
\partial_s q   =  ( \mathcal{L} + V)  q    +  B(q) + R(y,s) + G(w,W),   
\end{equation}
where
\begin{eqnarray}
\mathcal{L}  & =&  \Delta  - \frac{1}{2} y  \cdot  \nabla  + Id,\label{c4defini-ope-mathcal-L}\\
V(y,s) & = &        p \left(  \varphi^{p-1}(y,s)    -  \frac{1}{p-1}     \right) ,\label{c4defini-potential-V}  \\
B(q)  &    =&     \left(q +  \varphi  \right)^p   -   \varphi^p -  p \varphi^{p-1} q      , \label{c4defini-B-Q}\\
R(y,s) & = &  -   \partial_s\varphi +       \Delta \varphi - \frac{1}{2}  y \cdot \nabla \varphi - \frac{\varphi}{3}     + \varphi^p, \label{c4defini-rest-term}\\
G(w,W)  & =&  \left( \frac{1}{p-1}  \frac{ \bar \theta' (s)}{ \bar \theta (s)}  -e^{-s} \right) \left(  q+ \varphi \right) + F(w,W), \label{c4defini-N-term} 
\end{eqnarray}
and $F(w,W)$ is defined by \eqref{c4defini-F-1}. Let us remark  that  equation \eqref{c4equa-Q} is quite the same as in the  classical nonlinear heat equation apart for the extra term $G.$  Let us point out that this term has two important features.
On the one hand, it is a novel term, with respect to old literature (\cite{BKnon94} and \cite{MZdm97} in particular), very delicate to control, which makes our paper completely relevant. On the other hand, we will show in Lemma \ref{size-G} below that $G$ is exponentially small (in $s$), which means that its contribution will not affect the dynamics, which lay in $s^{-i}$ scales (with possible logarithmic corrections), as one may see from Definition \ref {defini-shrinking-set-S-t} of the shrinking set below, particularly item (i) with the estimates in the blowup region $P_1(t)$.\\
In the following, we recall some properties of the linear operator $\mathcal{L} $ and the potential $ V$. 

\begin{center}
   \textbf{Operator $\mathcal{L}$}
\end{center}
Operator $\mathcal{L}$ is  self-adjoint in  $\mathcal{D} (\mathcal{L}) \subset L^2_\rho (\mathbb{R}^N),$  where  $L^2_\rho(\mathbb{R}^N)$  defined  as  follows
$$   L^2_\rho  (\mathbb{R}^N) = \left\{        f \in L^2_{loc} (\mathbb{R}^N) \left| \right. \int_{\mathbb{R}^N}  |f(y)|^2 \rho (y) dy  < + \infty   \right\},$$
and
$$ \rho (y):  =  \frac{e^{-\frac{|y|^2}{4}}}{ (4\pi )^{\frac{N}{2}}}. $$
Besides, its spectrum set is explicitly given by
$$   \text{Spec} (\mathcal{L}) =  \left\{    1 -  \frac{m}{2}    \left|  \right.      m \in \mathbb{N}\right\}   .$$
Accordingly   to the eigenvalue $\lambda_m = 1  - \frac{m}{2}$, the correspond  eigen-space  is  given  by 
$$     \mathcal{E}_m =  \left<   h_{m_1} (y_1). h_{m_2} (y_2).... h_{m_N} (y_N)   \left|  \right.  m_1 + ...+ m_N = m    \right>, $$
where   $h_{m_i}$  is  the  (re-scaled) Hermite  polynomial in one  dimension. 

\medskip
- \textit{Key properties of potential $V$:} 

\begin{itemize}
\item[$(i)$] The  potential  $V(., s) \to 0$  in $L^2_\rho(\mathbb{R}^N)$ as  $s \to + \infty$:  In particular,    in the  region  $|y| \leq K_0 \sqrt s$ (the singular domain),     $V$ has  some weak perturbations    on   the   effect of   operator  $\mathcal{L}$.
\item[$(ii)$]   $V(y,s)$ is almost a constant  on the region $|y| \geq K_0 \sqrt s$:   For all  $\epsilon > 0$, there  exists   $ \mathcal{C}_\epsilon > 0$ and  $s_\epsilon$ such that 
$$     \sup_{ s \geq  s_\epsilon, \frac{|y|}{ \sqrt s}  \geq \mathcal{C}_\epsilon} \left|  V(y,s)   - \left( -\frac{p}{p-1}\right)  \right|   \leq  \epsilon. $$
Note  that   $ - \frac{p}{p-1 }  <   -1  $  and  that  the  largest  eigenvalue  of  $\mathcal{L}$ is  $1$. Hence, roughly speaking, we may assume that  $\mathcal{L} + V$ admits a  strictly  negative  spectrum. Thus,  we  can easily control   our solution in the  region $\{ |y| \geq K_0 \sqrt{s} \}$ with $K_0$ large enough.
 \end{itemize}

\medskip 
 From the preceding properties, it appears that the operator   of $\mathcal{L} + V$  does not share  the  same  behaviour inside  and outside   of  the  singular  domain $\{ |y| \leq   K_0 \sqrt{s}\}.$ Therefore, it is natural  to   decompose  every    $r  \in L^\infty(\mathbb{R}^N)$ as follows:
\begin{equation}\label{c4R=R-b+R-e}
r (y)=   r_b (y) +  r_e (y)  \equiv  \chi (y,s)   r(y) + (1- \chi (y,s) )    r (y),
\end{equation} 
where   $\chi (y,s)$  is defined as follows 
\begin{equation}\label{c4defini-chi-y-s}
\chi (y,s)  = \chi_0 \left(  \frac{|y|}{K_0 \sqrt s}   \right),
\end{equation}
recalling that  $\chi_0 $ is given by \eqref{c4defini-chi-0}.   From the above  decomposition, we immediately
have the following: 
\begin{eqnarray*}
\text{Supp }(r_b)  & \subset & \{ |y| \leq 2 K_0 \sqrt{s} \},\\
\text{Supp }(r_e) & \subset & \{ |y| \geq K_0 \sqrt{s}\}.
 \end{eqnarray*}
In the following  we  are interested in expanding    $r_b$  in   $L^2_\rho \left( \mathbb{R}^N \right)$ according to the basis  which  is created by the eigenfunctions of  operator $\mathcal{L}$:
\begin{eqnarray*}
r_b (y)  & =  &r_0  + r_1 \cdot y +   y^T \cdot r_2 \cdot y  - 2 \text{ Tr}(r_2)  + r_-(y),
\end{eqnarray*}
 or
\begin{eqnarray*}
r_b (y)  & =  &r_0  + r_1 \cdot y +r_\perp(y),
\end{eqnarray*}
where  
\begin{equation}\label{c4defini-R-i}
r_i =   \left(   P_\beta ( r_b )  \right)_{\beta \in \mathbb{N}^N, |\beta|= i}, \forall  i \geq 0,
\end{equation}
with $P_\beta(r_b)$  being  the projection  of  $r_b$ on   the   eigenfunction    $h_\beta$ defined as follows:
\begin{equation}\label{c4defin-P-i}
P_\beta (r_b) =  \int_{\mathbb{R}^N}  r_b   \frac{h_\beta}{\|h_\beta\|_{L^2_\rho(\mathbb{R}^N)}}   \rho dy, \forall \beta \in \mathbb{N}^N.
\end{equation}
Besides that, we also dennote
\begin{equation}\label{c4defini-R-perp}
r_\perp =  P_\perp (r)  =   \sum_{\beta \in \mathbb{N}^N, |\beta| \geq 2}   h_\beta P_\beta (r_b),
\end{equation}
and
\begin{equation}\label{c4defini-R--}
r_-   =   \sum_{\beta \in \mathbb{R}^N, |\beta| \geq 3}   h_\beta P_\beta (r_b).
\end{equation}
In other  words,  $r_\perp $ is   the part  of  $r_b$  which    is  orthogonal   to the   eigenfunctions   corresponding to   eigenvalues $0$ and  $1$  and  $r_-$  is  orthogonal   to the  eigenfunctions   corresponding to   eigenvalues $ 1, \frac{1}{2}$ and  $0$.     We  should note that  $r_0$ is a scalar,  $r_1$ is a vector  and  $r_2$ is a  square  matrix   of   order $N;$ they are all components of $r_b$ and not $r$. Finally,   we   write   $r$  as follows
\begin{eqnarray}
r (y) &  =  &  r_0 + r_1 \cdot y  +  y^T \cdot r_2 \cdot y  - 2 \text{ Tr}( r_2)   + r_- + r_e (y).\label{c4represent-non-perp}
\end{eqnarray}
or
\begin{eqnarray}
r (y) &  =  &  r_0 + r_1 \cdot y  + r_\perp(y) + r_e (y).\label{c4represent-non-perp1}
\end{eqnarray}

\subsection{ Localization  variable} In this part,  we will state the rigorous  form of problem \eqref{c4rescaled-function-U}  in region $P_2$.  Using  the equation of \eqref{equa-u-non-local} then $\mathcal{U}$ defined by  \eqref{c4rescaled-function-U} satisfies:
\begin{equation}\label{equa-rescaled-rigogous}
    \partial_t  \mathcal{U} =  \Delta_\xi \mathcal{U}   + \mathcal{U}^p   + \left( \frac{1}{p-1}   \frac{\tilde{\theta}'_\tau(\tau)}{\tilde{\theta}(\tau)} - \rho(x) \right) \mathcal{U},
\end{equation}
where 
\begin{equation}\label{defi-tilde-theta-tau}
    \tilde{\theta}(\tau) = \theta(t) = \theta(\tau \varrho(x) +t(x)), \text{ and } \rho(x) \text{ defined as in } \eqref{c4defini-theta}.
\end{equation}

\subsection{Shrinking set }
Below, we  aim  to construct a   special set where the behaviour of the solution $U$ of equation \eqref{equa-U-theta-'-theta} can be controlled.

\begin{definition}[Definition of  $S(t)$]\label{defini-shrinking-set-S-t} Let us consider positive constants   $K_0, \epsilon_0, \alpha_0, A, \delta_0, C_0, \eta_0$ and  take $t \in [0, T)$ for some  $T>0.$ Then,  we introduce  the following  set    
$$ S(  K_0 ,\epsilon_0, \alpha_0, A, \delta_0, C_0, \eta_0, t) \quad   (S(t) \text{ in short}),$$
as the subset  of $C^{2}\left( \Omega    \right) \cap C(   \bar \Omega ),$  containing  all functions  $U$  satisfying the following conditions:
\begin{itemize}
\item[$(i)$] \textbf{Estimates  in  $P_1(t)$}: We have  that the function  $q$ introduced   in  \eqref{defi-q=w-varphi} belongs to $V_{ A} (s),  s = -  \ln(T-t),$ where    $V_{ A} (s)$ is   a  subset of  all functions $r$  in $L^\infty (\mathbb{R}^N),$  satisfying the following estimates:
\begin{eqnarray*}
 &&| r_i   |    \leq   \frac{A}{s^{2}},  (i=0,1),  \text{ and  }  | r_2 | \leq  \frac{A^2 \ln s}{s^2},\\
&&|r_-(y)|  \leq    \frac{A^2}{s^2} (1 + |y|^3)   \text{ and }  \left|   \left(   \nabla r  \right)_\perp    \right|   \leq  \frac{A}{s^2} (1 + |y|^3),  \forall y \in \mathbb{R}^N,\\
&&\| r_e \|_{L^\infty(\mathbb{R}^N)} \leq    \frac{A^2}{ \sqrt s},
\end{eqnarray*}      
where the definitions  of   $r_i $, $ r_-, (\nabla r)_\perp$ and $r_e$   are given by \eqref{c4defini-R-i}, \eqref{c4defini-R-perp}, \eqref{c4defini-R--}, and  \eqref{c4R=R-b+R-e}, respectively, depending on the similarity variables defined by \eqref{similarity-variable}.
\item[$(ii)$] \textbf{ Estimates in  $ P_2 (t)$:}     For all   $|x|    \in \left[  \frac{K_0}{4 }  \sqrt{(T-t)|\ln(T-t)|},  \epsilon_0 \right], \tau (x,t)  =   \frac{t- t(x)}{\varrho (x)}$    and $|\xi|  \leq  \alpha_0  \sqrt{|\ln \varrho(x)|},$    we have  the following
\begin{eqnarray*}
\left|   \mathcal{U}  (x, \xi, \tau(x,t))  - \hat{\mathcal{U}}(\tau(x,t))  \right|  & \leq &   \delta_0, \\
\left|  \nabla_\xi  \mathcal{U}  (x, \xi, \tau(x,t))  \right|  & \leq  & \frac{C_0}{\sqrt{|\ln \varrho(x)|}},\\
\end{eqnarray*}
where   $\mathcal{U},$  $\hat{\mathcal{U}}$ and  $\varrho(x) $  defined as in \eqref{c4rescaled-function-U},  \eqref{c4defini-theta} and \eqref{c4defin-hat-mathcal-U-tau},   respectively.
\item[$(iii)$] \textbf{Estimates in  $P_3(t)$:}  For all  $  x \in  \{ |x| \geq \frac{\epsilon_0}{4} \} \cap  \Omega,$ we have 
\begin{eqnarray*}
\left|    U(x,t)  -  U(x,0)  \right|  & \leq  &  \eta_0,\\
\left|  \nabla U(x,t) -  \nabla e^{t \Delta} U(x, 0) \right| & \leq  &  \eta_0. 
\end{eqnarray*}
\end{itemize}
\end{definition}
Using the definition of $S(t)$, we have the following 
\begin{lemma}[Growth estimates for  $q$ belonging to $V_{A}(s)$] \label{lemma-properties-V-A-s}  Let us consider  $K_0 \geq 1$ and  $A \geq 1$. Then, there exists $s_1=s_1(A,K_0)$ such that for all $ s \geq s_1$ and  $q\in V_{A}(s),$ we have the following estimates:  
$$   |q (y,s)| \leq  \frac{C (K_0)A^2 }{\sqrt s }   \text{ and }    \left| q (y,s)\right| \leq   \frac{C (K_0) A^2 \ln s}{s^2} (1 +   |y|^3). $$
In particular, we  obtain  the following  improved  estimate in the blowup region
\begin{eqnarray*}
\|  q\|_{L^\infty}({\{|y| \leq K_0 \sqrt{s} \})}  \leq  \frac{C(K_0) A  }{\sqrt{s}}  .
\end{eqnarray*}
\end{lemma}
\begin{proof}
The proof immediately arises by adding the given bounds in Definition \ref{defini-shrinking-set-S-t}. 
\end{proof}

\begin{remark}[Universality constant]
In our proof, we  introduce a lot of parameters, then, the universality  upper bound will depend on these parameters. for more convenience, from now on, we denote the universality constant by $C$ as long as it  only depends on $K_0, \Omega, r, \gamma, p$ and  $N$, intrinsic constants. However, once it depends more on extra-parameters, example $C_2$ as in Proposition \ref{propo-bar-mu-bounded} below, we will write 
$$C(C_2).$$
\end{remark}

\medskip
Next  we  derive the  dynamic  of $ \theta$ defined as in  \eqref{defi-theta-by-U}: 
\begin{proposition}[Dynamic of $\theta$]\label{propo-bar-mu-bounded} Let us consider  \eqref{subcritical-condition}
with   $\Omega$ a  bounded domain with smooth boundary, and $C_2>0$.  Then, there  exists  $K_{2}> 0$ such that  for all  $K_0 \geq K_{2},  \delta_{0 } > 0,$ 
we can find     $\epsilon_{2} (K_0, \delta_{0}, C_2) > 0$ such that   for all   $\epsilon_0  \leq \epsilon_2$  and  $A \geq 1, C_0 > 0, \eta_0 > 0$, and $\alpha_0 >0 $,  there exists $T_{2}   > 0$ such that  for  all  $T  \leq T_2$ the following holds:  Assuming   $U$ is a non negative  solution  of   equation \eqref{equa-U-theta-'-theta}    on $[0, t_1],$ for some $t_1 < T$ and  $U(t) \in S(K_0, \epsilon_0, \alpha_0, A, \delta_0, C_0,\eta_0,t)= S(t) $ for all  $t \in [0, t_1]$ with initial data $U(0)=U_0$  satisfying the following estimate
\begin{equation}\label{condition-integral-U-0}
  \frac{1}{C_2} \leq   \def\avint{\mathop{\,\rlap{--}\!\!\int}\nolimits} \avint_\Omega
U^r_0\,dx  \leq C_2,
\end{equation}
then there hold:
\begin{itemize}
\item[$(i)$]  For all $t\in [0,t_1],$ we have   $\theta (t) > 0$ and 
 \begin{eqnarray}
\frac{1}{C(C_2)} \leq \theta  (t)  \leq   C(C_2). \label{bound-bar-theta}
\end{eqnarray}

\noindent
Moreover, there exists $\varepsilon := \varepsilon(N,r,p) >0$ such that  for all $t \in (0,t_1)$, we have
\begin{eqnarray}
\left| \theta' (t) \right|   & \leq  &   (T-t)^{-1 + \varepsilon} \label{bound-bar-theta-'}.
\end{eqnarray}  
 \item[$(ii)$] In particular, if  $U \in S(t)$  for all   $ t \in [0, T)$, then   there exists a constant $\theta^*(\Omega, r,p,N,\gamma, U(0))>0$ such that
 $$ \theta(t)  \to \theta^* \text{ as } t \to T.$$
\end{itemize} 
\end{proposition}
\begin{proof}
We kindly refer the readers to see the details of the  proof in Appendix \ref{appexidix-propo-theta}.
\end{proof}

\subsection{Constructing initial data } 
In  this  part, we want to  build initial data $U_0 \in S(0)$ for equation \eqref{equa-U-theta-'-theta}. To this end we follow a similar approach with one developed in \cite{DZM3AS19}. Firstly, we introduce  the following cut-off function  
\begin{eqnarray}
\chi_1 (x)   & = &  \chi_0 \left(   \frac{|x|}{ \sqrt T |\ln T|}   \right),  \label{c4defini-chi-1} 
\end{eqnarray}
where  $\chi_0$ defined in \eqref{c4defini-chi-0}. Next,  we  introduce  $H^*$ as a suitable modification of the final asymptotic profile in the intermediate region:   
\begin{equation}\label{c4defini-H-epsilon-0}
H^*  (x)  =  \left\{   \begin{array}{rcl}
&  &  \left[  \displaystyle \frac{(p-1)^2}{8p} \frac{|x|^2}{|\ln|x||}    \right]^{ -\frac{1}{p-1}}, \quad  \forall   |x| \leq  \min\left(  \frac{1}{4} d (0 ,\partial \Omega), \frac{1}{2}  \right), x \neq 0, \\[0.7cm]
&  & 1,   \quad   \forall   |x| \geq  \frac{1}{2}  d(0, \partial \Omega)  .\\[0.5cm]
\end{array}      \right.
\end{equation}
Now for any  $(d_0, d_1 ) \in   \mathbb{R}^{1 + N}$, we define initial data
\begin{eqnarray}
U_{d_0, d_1}(x,0)  & = &    H^*(x) \left(  1 -  \chi_1 (x)    \right) ,\nonumber\\
& = & T^{-\frac{1}{p-1}}  \left[  \varphi \left( \frac{x}{\sqrt{T}}, - \ln s_0  \right)   +  (d_0 + d_1 \cdot z_0)  \chi_0 \left(  \frac{|z_0|}{\frac{K_0}{32}}\right) \right]  \chi_1 (x) \label{c4defini-initial-data},
\end{eqnarray}
where   $ z_0 = \displaystyle  \frac{ x}{ \sqrt{T  |\ln T|}} , s_0 = -\ln T$;     $\varphi, \chi_0,   \chi_1 $ and $ H^*  $ are defined as in  \eqref{defi-varphi}, \eqref{c4defini-chi-0},  \eqref{c4defini-chi-1} and  \eqref{c4defini-H-epsilon-0}, respectively.  

\medskip
Note that we can also  write the initial data in  the variable similarity \eqref{similarity-variable}
corresponding $y_0 = \frac{x}{\sqrt{T}} \in \Omega_{s_0}$ and $s_0 = - \ln T $.

- For $W(y_0,s_0)$: 
\begin{equation}\label{defi-W-s-0}
\begin{array}{rcl}
     W(y_0,   s_0) &=& \left[   \varphi(y_0,s_0) + (d_0 + z_0 \cdot d_1 ) \chi_0 \left(  \frac{|z_0|}{\frac{K_0}{32}}\right)   \right]\chi_1( y_0 \sqrt{T}) \\[0.2cm]
     &+&  T^{\frac{1}{p-1}} H^*(y_0 \sqrt{T}) (1- \chi_1 (y_0 \sqrt{T})).
\end{array}
\end{equation}
- For $w(y_0, s_0)$
\begin{equation}\label{defi-w-s-0}
\begin{array}{rcl}
     w(y_0,   s_0) &=& \left\{\left[   \varphi(y_0,s_0) + (d_0 + z_0 \cdot d_1 ) \chi_0 \left(  \frac{|z_0|}{\frac{K_0}{32}}\right)   \right]\chi_1( y_0 \sqrt{T}) \right. \\[0.2cm]
     &+& \left.  T^{\frac{1}{p-1}} H^*(y_0 \sqrt{T}) (1- \chi_1 (y_0 \sqrt{T})) \right\} \psi_{M_0}(y_0,s_0),
\end{array}
\end{equation}
for all $y_0 \in \Omega_{s_0}.$ We should point out that $\psi_{M_0}$ defined as in \eqref{c4defini-psi-M-0-cut} and $w$ vanishes  outside $\Omega_{s_0}$. 

- For  $q(y_0,s_0)$: 
\begin{eqnarray}\label{defi-q-y-0-s-0}
q(y_0,s_0)     =   w(y_0, s_0) - \varphi(y_0,s_0).     
\end{eqnarray}
In the following, we construct appropriate initial data of the form \eqref{c4defini-initial-data}, i.e. initial data which belong to the shrinking set $S(0).$
\begin{proposition}[Construction of initial data]\label{c4proposiiton-initial-data} There  exists $K_{3}> 0$ such that  for all  $K_0 \geq K_{3} $ and $  \delta_{ 3 } > 0,$ there exist $\alpha_{3} (K_0, \delta_{3}) > 0 $  and $C_{3}(K_0) > 0$ such that for  every  $\alpha_0 \in   (0, \alpha_{3}],$     there exists  $\epsilon_{3} (K_0, \delta_{3}, \alpha_0) > 0$ such that   for every  $\epsilon_0 \in (0,\epsilon_{3}]$  and  $A \geq 1,$ there exists  $T_{3} (K_0,  \delta_{3}, \epsilon_0, A,C_3 ) > 0 $  such that  for all  $T \leq    T_3$ and $ s_0 = - \ln T,$ then  the following  hold:

(I)  We can find a  set  $\mathcal{D}_{ A}     \subset [-2,2] \times [-2,2]^N$  such that   if we define the following mapping
\begin{eqnarray*}
\Gamma : \mathbb{R} \times \mathbb{R}^N     & \to &   \mathbb{R} \times \mathbb{R}^N\\
(d_0, d_1)   & \mapsto   &  \left( q_0, q_1  \right)(s_0),
\end{eqnarray*}
 then, $\Gamma$  is  affine,  one to  one   from   $\mathcal{D}_{A}$   to 	$\hat{\mathcal{V}}_A (s_0),$ where $\hat{\mathcal{V}}_A (s)$ is defined as  follows
\begin{equation}\label{c4defini-hat-mathcal-V-A}
\hat{\mathcal{V}}_A (s) =  \left[-\frac{A}{s^2}, \frac{A}{s^2} \right]^{1+N}.
\end{equation}  
Moreover,  we have     
$$\Gamma \left|_{\partial \mathcal{D}_{A}}  \right.   \subset   \partial \hat{\mathcal{V}}_A (s_0), $$
and 
\begin{equation}\label{c4deg-Gamma-1-neq-0}
 \text{deg} \left(  \Gamma \left. \right|_{ \partial  \mathcal{D}_{A} }  \right) \neq 0,
\end{equation}
  where   $q_0,q_1$  are defined  as  in \eqref{c4represent-non-perp} and  considered as the  components of $q_{d_0,d_1} (s_0)$  given  by \eqref{defi-q-y-0-s-0}.

(II) For all   $ (d_0, d_1) \in \mathcal{D}_{ A}$ then, initial data   $U_{d_0,d_1}$ defined in  \eqref{c4defini-initial-data}   belongs to  
$$S(K_0, \epsilon_0, \alpha_0, A, \delta_{3},   C_{3}, 0, 0)= S(0),$$
 where  $S( 0) $ is defined  in Definition \ref{defini-shrinking-set-S-t}.   Moreover, the following   estimates  hold
\begin{itemize}
\item[$(i)$] Estimates in  $P_1(0):$   We have   $q_{d_0,d_1} (s_0) \in \mathcal{V}_{K_0, A} (s_0)$ and  
$$ |q_0 (s_0)| \leq  \frac{A}{s_0^2}, \quad   |q_{1,j}  (s_0)| \leq \frac{A}{s_0^2}, \quad  |q_{2,i,j}(s_0)| \leq \frac{\ln s_0}{s_0^2},   \forall i,j \in \{1,..., N\},$$
$$ |q_- (y,s_0) | \leq 	\frac{1}{s_0^2} (|y|^3 + 1), \quad  |\nabla q_\perp (y,s_0) | \leq \frac{1}{s_0^2} (|y|^3 + 1),  \forall y \in \mathbb{R}^N,$$
and
$$  \|q_e   \|_{L^\infty}  \leq \frac{1}{\sqrt{s_0}}  
\text{ and } \|\nabla q (s_0)\|_{L^\infty} \leq \frac{1}{\sqrt{s_0}}  , $$
 where  the  components  of $q_{d_0, d_1} (s_0)$  are defined in  \eqref{c4defini-R-perp}.
\item[$(ii)$] Estimates  in  $P_2(0)$:  For  every $|x| \in \left[  \frac{K_0}{4} \sqrt{T |\ln T|}  , \epsilon_0 \right] , \tau_0(x) = - \frac{ t(x) }{ \varrho (x)}$  and $|\xi|  \leq 2\alpha_0 \sqrt{|\ln \varrho(x)|},$ we have
$$   \left|  \mathcal{U} (x, \xi, \tau_0(x) )    - \hat{\mathcal{U}} (\tau_0(x))  \right| \leq   \delta_{3} \text{ and }    |\nabla_\xi \mathcal{U} (x, \xi, \tau_0(x))| \leq \frac{C_{3}}{\sqrt{|\ln \varrho(x)|}} ,
 $$
\end{itemize}
where   $\mathcal{U}, \hat{\mathcal{U}}, $ and  $\varrho(x)$ are  defined  as in  \eqref{c4rescaled-function-U}, \eqref{c4defin-hat-mathcal-U-tau} and  \eqref{c4defini-theta}, respectively.
 \end{proposition}
\begin{proof}
Notably,  the shrinking set and initial data are the same as in \cite{MZnon97}, therefore for the proof of the current lemma we  kindly refer the reader to check \cite[Lemma 2.4]{MZnon97}.   
\end{proof}

\section{Proof of  Theorem  \ref{theorem-existence}}
This section, is devoted to the proof of our main result Theorem \ref{theorem-existence}. Its proof though is quite technical, and it requires the following auxiliary result. Then the proof of   Theorem \ref{theorem-existence} will be given in the end of the current section. 
\begin{proposition}[A solution  within  $S(t)$]\label{c4proposition-existence-U}  There exist  positive parameters  $T, K_0, \epsilon_0,$ $\alpha_0, A, \delta_0, C_0,  \eta_0$ such that  we can choose  a  couple  $(d_0, d_1) \in \mathbb{R}^{1 +N}$
 such that  the solution  $U$ of equation \eqref{equa-U-theta-'-theta} defined  on  $ \Omega \times  [0,T)$ and with initial data $U_{d_0, d_1}$ given  in \eqref{c4defini-initial-data} belongs to   
 $S(t) $  for all $ t \in [0,T))$, where $S(t)$ introduced in Definition \ref{defini-shrinking-set-S-t}.
\end{proposition}
\begin{proof}
This proposition plays a vital role in proving Theorem \ref{theorem-existence}.  However,  the conclusion is very classical and stems from a robustness argument is encountered in \cite{BKnon94}, \cite{MZdm97}, \cite{MZnon97}. In particular, we refer the readers to \cite[Proposition 4.4]{DZM3AS19}. In particular, the argument consists of the following two main steps:

- Step 1:   We reduce  our problem  to a   finite dimensional one. More  precisely,   we  prove that the task of controlling the asymptotic behaviour of  $U(t) \in  S(t)$ for all $t \in [0,T)$ is reduced to governing the behaviour of $(q_0,q_1) (s)$ in  $\hat{\mathcal{V}}_A (s)$ (see Proposition  \ref{c4proposition-reduction-finite} below).

- \textit{Step 2:}   In this step, we aim at  proving the existence of a $(d_0, d_1) \in \mathbb{R}^{1 + N}$ such that the solution of  the solution  $U$ of equation \eqref{equa-U-theta-'-theta}  and with initial data $U_{d_0, d_1}$ given  in \eqref{c4defini-initial-data} belongs to  $S(t)$ with suitable parameters.  Then, the conclusion follows from a topological argument based on Index theory. This completes the proof of the proposition.

\end{proof}
\begin{center}
    \textbf{Conclusion of Theorem \ref{theorem-existence}}
\end{center}
Let us fix positive parameters  $T, K_0, \epsilon_0, \alpha_0, A, \delta_0, C_0,\eta_0$ such that Propositions \ref{propo-bar-mu-bounded}, and  \ref{c4proposition-existence-U} hold true. Then, we obtain 
$$ U(t) \in S(t) \quad  \forall t \in [0,T).$$
Using item $(ii)$ of Proposition \ref{propo-bar-mu-bounded}, we derive
$$ \theta(t) \to \theta^*, \text{ as } t \to T, $$
for some $\theta^*= \theta^*(\Omega, p, \gamma,r) >0$. In particular, we also obtain the fact that the constructed solution satisfied 
$$ \left| \theta'(t) \right| \leq (T-t)^{-1 + \varepsilon },  $$
which yields
\begin{equation}\label{the-error-of theta}
    \left| \theta(t) - \theta^* \right| \leq (T-t)^{\varepsilon}.
\end{equation}

\medskip
Now we are ready to proceed with the proof of Theorem \ref{theorem-existence} which follows the same lines as in  \cite[pages 1306-1310]{DZM3AS19}. However, for readers convenience we present in the following the key estimates, whilst further details can be found in the aforementioned work.

- Proof of item $(i)$: By virtue of   in Definition \ref{defini-shrinking-set-S-t} $(i), (iii)$  and  equality  \eqref{c4defini-w-small}, we obtain
\begin{equation*}
    \left\| W(y,s) - \left( p-1 + \frac{(p-1)^2}{4p} \frac{|y|^2}{s} \right)^{-\frac{1}{p-1}} \right\|_{L^\infty(\Omega_s)} \leq \frac{C}{1 + \sqrt{s}},
\end{equation*}
which via \eqref{similarity-variable} yields 
\begin{eqnarray}
 & &\left\| (T-t)^{\frac{1}{p-1}}U(.,t) -  \left( p-1 + \frac{(p-1)^2}{4p} \frac{|.|^2}{ (T-t) |\ln(T-t)|} \right)^{-\frac{1}{p-1}} \right\|_{L^\infty(\Omega)} \label{estima-U-in chapter-5} \\
 & & \leq \frac{C}{1 + \sqrt{|\ln(T-t)|}}.\nonumber
\end{eqnarray}
From \eqref{defi-U-by-u},  \eqref{the-error-of theta} and \eqref{estima-U-in chapter-5},  we  directly  derive \eqref{estima-T-t-u-theorem-varphi-0}. The uniqueness of the blowup point follows by the result \cite[Theorem 2.1]{GKcpam89} in applying \eqref{estima-U-in chapter-5}, and for the complete argument for this point, the readers can see the same one in   . Finally the existence of the blowup profile $u^*$   is quite the same as in \cite[Proposition 3.5]{DZM3AS19}. This concludes  the proof of Theorem \ref{theorem-existence}.

\section{Behavior of $L^k$ norm at blowup time}
This part is devoted to the proof of Corollary \ref{corollary-u-L-k}.  Let us show some useful estimates for proof.

\begin{lemma}\label{lemma-t(x)} 
Let us consider  $t(x)$, defined as in \eqref{c4defini-t(x)-} for all $|x| \leq \epsilon_0$, and $\rho(x) =  T - t(x)$. Then, we have
$$ \rho(x) \sim \frac{8}{K_0^2} \frac{|x|^2}{|\ln|x||},  $$
and 
$$ \ln \rho (x) \sim 2 \ln|x|,$$
as $ x \to 0$.
\end{lemma}
\begin{proof}
The proof is straightforward by definition  \ref{corollary-u-L-k} of $t(x)$.
\end{proof}
Now, we aim to give some estimate on $U$, trapped in $S(t)$. 
\begin{lemma}\label{lemma-estimate-U-x-t-in-S-t}
Let us consider $U(t ) \in S(K_0, \epsilon_0, \alpha_0, A, \delta_0, C_0, \eta_0, t)$ for all $t \in [0,T)$, defined as  in Definition \ref{defini-shrinking-set-S-t}  where  $ \eta_0 \ll 1$, and $K_0' \geq K_0$.   Then, the following hold:
\begin{itemize}
    \item[$(i)$] For all $|x| \leq \frac{K_0'}{4} \sqrt{(T-t) |\ln(T-t)|}$, we have 
    \begin{eqnarray*}
    \left| U(x,t ) - (T-t)^{-\frac{1}{p-1}} \varphi_0\left( \frac{|x|}{\sqrt{(T-t)|\ln(T-t)|}}\right) \right|\leq \frac{C(A)(T-t)^{-\frac{1}{p-1}}}{1 + \sqrt{|\ln(T-t)|}},
    \end{eqnarray*}
 where $\varphi_0,$ defined \eqref{defi-varphi-0}, and gradient's estimate
  \begin{eqnarray*}
    \left| \nabla_x U(x,t )  \right|\leq \frac{C( A)(T-t)^{-\frac{1}{p-1}-\frac{1}{2}}}{1 + \sqrt{|\ln(T-t)|}}.
    \end{eqnarray*}
    \item[$(ii)$] For all $|x| \in \left[ \frac{K_0'}{4} \sqrt{(T-t) |\ln(T-t)|}, \epsilon_0 \right]$, we have
    \begin{eqnarray*}
    |  U(x,t)| \leq 4 \left( \frac{(p-1)^2}{8p} \frac{|x|^2}{|\ln|x||} \right)^{-\frac{1}{p-1}},
    \end{eqnarray*}
    and
    \begin{eqnarray*}
    | \nabla_x U(x,t) | \leq 8 C_0 \left( \frac{(p-1)^2}{8p} \frac{|x|^2}{|\ln|x||} \right)^{-\frac{1}{p-1} -\frac{1}{2}} \frac{1}{|\ln|x||},
    \end{eqnarray*}
    provided that $K_0 \geq K_6$ and $ \epsilon_0 \leq \epsilon_6(K_0),$. 
    \item[$(iii)$] For all $|x| \geq \epsilon_0$, we have
    $$ \frac{1}{2} \leq  U(x,t)  \leq C(\epsilon_0,\eta_0),$$
    and
    $$ | \nabla_x U(x,t)| \leq C(\eta_0, \epsilon_0),$$
    provided that $\eta_0 \ll 1 $.
\end{itemize}
\end{lemma}
\begin{proof} The proof is mainly based on estimates in  Definition \ref{defini-shrinking-set-S-t}.

- The proof of item $(i)$:  Let us consider  $|x| \leq \frac{K_0'}{4} \sqrt{(T-t)|\ln(T-t)|}$, where $K_0' \geq K_0$.  Once  $ \epsilon_0 \leq \frac{1}{M_0} ,$ where $M_0$ introduced in \eqref{c4defini-psi-M-0-cut}, we have
$$ W(y,s) = w(y,s), \text{ with } y = \frac{x}{\sqrt{T-t}} \text{ and } s = -\ln(T-t).$$
Using \eqref{similarity-variable}, we derive
\begin{eqnarray*}
(T-t)^{-\frac{1}{p-1}} U(x,t) = W(y,s) = \varphi(y,s) +q(y,s).
\end{eqnarray*}
Then, from Lemma \ref{lemma-properties-V-A-s},  we get
$$ | q(y,s)| \leq \frac{C A^2}{\sqrt{s}}. $$
By definition of $\varphi$, defined in \eqref{defi-varphi}, we conclude the first estimate. In addition to that, the second one is proved by the same technique.

- The proof of item $(ii)$: We now consider $ |x| \in \left[ \frac{K_0'}{4} \sqrt{(T-t)|\ln(T-t)|}, \epsilon_0 \right] \subset P_2(t)$, recall  item $(ii)$ in  Definition \ref{defini-shrinking-set-S-t}, we have
\begin{eqnarray*}
\left|  \mathcal{U}(x,0,\tau(x,t)) - \left( (p-1) (1 - \tau(x,t)) + \frac{(p-1)^2}{4p} \frac{K^2_0}{16} \right)^{-\frac{1}{p-1}}  \right| \leq \delta_0,
\end{eqnarray*}
and
\begin{equation}\label{estimate-nable-U-x-0}
    | \nabla_\xi \mathcal{U}(x,0,\tau(x,t))| \leq  \frac{C_0}{ |\ln(\rho (x))|}.
\end{equation}
We also have the fact that 
$$ 1 -\tau(x,t) = \frac{t-t(x)}{T-t(x)} \in [0,1].   $$
Then, it follows
$$ \left( (p-1) (1 - \tau(x,t)) + \frac{(p-1)^2}{4p}) \frac{K^2_0}{16} \right)^{-\frac{1}{p-1}} \leq \left(\frac{(p-1)^2}{4 p} \frac{K_0^2}{16} \right)^{-\frac{1}{p-1}}. $$
Taking $\delta_0 \leq \left(\frac{(p-1)^2}{4 p} \frac{K_0^2}{16} \right)^{-\frac{1}{p-1}}$,  we derive
\begin{eqnarray*}
| \mathcal{U}(x,0,\tau(x,t))| \leq 2 \left(\frac{(p-1)^2}{4 p} \frac{K_0^2}{16} \right)^{-\frac{1}{p-1}}.
\end{eqnarray*}
Using the fact that, 
$$ U(x,t) = (T-t(x))^{-\frac{1}{p-1}} \mathcal{U}(x,0,\tau(x,t)),  $$
then, we have
$$ | U(x,t) | \leq 2 \left(\frac{(p-1)^2}{4 p} \frac{K_0^2}{16} \right)^{-\frac{1}{p-1}}  (T-t(x))^{-\frac{1}{p-1}}.$$
Apply Lemma \ref{lemma-t(x)} with $\epsilon_0 \leq \epsilon_6(K_0)$,  we obtain
\begin{eqnarray*}
|U(x,t) | \leq   4 \left( \frac{(p-1)^2}{8p} \frac{|x|^2}{|\ln|x||} \right)^{-\frac{1}{p-1}}.
\end{eqnarray*}
We conclude the first one, and the second one follows \eqref{estimate-nable-U-x-0} and a same technique.  In particular,  item $(iii)$ directly follows from the third  one  in Definition \ref{defini-shrinking-set-S-t}.
This completes the proof of Lemma \ref{corollary-u-L-k}. 
\end{proof}

Now, we produce the proof of Corollary \ref{corollary-u-L-k}.

\begin{proof}[Proof of Corollary  \ref{corollary-u-L-k}]
We will consider the two cases mentioned in the statement.\\
- The case  where $ \frac{k}{p-1} - \frac{N}{2} < 0$, we decompose as follows
\begin{eqnarray}
\|U\|^k_{L^k(\Omega)} &=& \int_{\Omega} U^k dx = \int_{|x| \leq \frac{K_0}{4} \sqrt{ (T-t) |\ln(T-t)|} }  U^k dx \label{decompose-U-L-k} \\
& + &       \int_{ \frac{K_0}{4} \sqrt{ (T-t) |\ln(T-t)|} \leq |x| \leq \epsilon_0 }  U^k dx +\int_{ \{|x| \geq \epsilon_0\} \cap \Omega} U^k dx.\nonumber
\end{eqnarray}
And, we also have a fundamental integral: for all $k >0$ and $K_0'$
\begin{eqnarray}
& & \int_{ \frac{K_0'}{4}\sqrt{(T-t)|\ln(T-t)|} \leq |x| \leq \epsilon_0} \left(\frac{|x|^2}{|\ln|x||} \right)^{-\frac{k}{p-1}} dx \label{integral-fundamental} \\
 &=& \left( \frac{K_0'}{4}\right)^{N -\frac{2k}{p-1}} \frac{(T-t)^{\frac{N}{2} - \frac{k}{p-1}}|\ln(T-t)|^\frac{N}{2}(1 + o_{t\to T}(1)) }{2\left( \frac{2k}{p-1} - N\right)} + f(\epsilon_0),\nonumber
\end{eqnarray}
where $f$ is some regular function. 

\medskip
\noindent
In addition to that, we also have
\begin{eqnarray}
& &(T-t)^{-\frac{k}{p-1}}\int_{|x| \leq \frac{K_0'}{4} \sqrt{(T-t)|\ln(T-t)|}} \varphi_0^k \left( \frac{|x|}{\sqrt{(T-t)|\ln(T-t)|}}\right) dx \label{integral-P-1-K-0'} \\
& = & (T-t)^{\frac{N}{2} - \frac{k}{p-1}}|\ln(T-t)|^\frac{N}{2} \int_0^\frac{K_0'}{4} \varphi_0(r) r^{N-1} dr.\nonumber
\end{eqnarray}
Then, using \eqref{lemma-estimate-U-x-t-in-S-t}, \eqref{integral-fundamental}, and \eqref{integral-P-1-K-0'},  with $ K_0' = K_0 $,    we derive
$$ \|U(t)\|_{L^k(\Omega)}^k \leq C( A, \epsilon_0, \eta_0). \forall t \in [0,T).$$
We now aim to prove that 
$$ \lim_{t \to T} \|U(t) \|_{L^k(\Omega)}^k   < +\infty.$$
It is sufficient  to  prove
$$  \left| \partial_t \|U(t) \|_{L^k(\Omega)}^k   \right| \leq  C(A, \epsilon_0, \eta_0) (T-t)^{-1+\varepsilon},  $$
with a small positive  $\varepsilon \in (0,1) $. Indeed, using \eqref{equa-U-theta-'-theta}, we derive
\begin{eqnarray*}
 \int_\Omega     U^{k-1} U_t  dx = \int_\Omega \Delta U U^{k -1} + \int_{\Omega}  U^{  p-1 + k  }    + \left( \frac{1}{p-1} \frac{\theta'(t)}{\theta(t)}  - 1\right)  \int_{\Omega} U^k.
 \end{eqnarray*}
 By  \eqref{bound-bar-theta-'}, we obtain, 
 $$ \left| \left( \frac{1}{p-1} \frac{\theta'(t)}{\theta(t)}  - 1\right)  \int_{\Omega} U^k   \right| \leq C( A, \epsilon_0, \eta_0) (T-t)^{-1+\varepsilon_1}, \text{ with } \varepsilon_1 \in (0,1).$$
 It remains to estimate for the first and second ones. Estimate for $  \int_\Omega \Delta U U^{k -1}$: Using the integration by parts, we have 
$$ \int_\Omega \Delta U U^{k -1} =  (1-k) \int_\Omega | \nabla U|^{2} U^{k -2}. $$ 
 By using  Lemma \ref{lemma-estimate-U-x-t-in-S-t}, \eqref{integral-fundamental}, \eqref{integral-P-1-K-0'}, and a similar decomposition as in \eqref{decompose-U-L-k},    we estimate
 $$ \int_\Omega | \nabla U|^{2} U^{k -2} \leq C( A, \epsilon_0, \eta_0 ) ( 1 + (T-t)^{\frac{N}{2} -\frac{k}{p-1} -1}), $$
 Similarly, we obtain
 $$  \int_{\Omega}  U^{  p-1 + k  }   \leq C( A, \epsilon_0, \eta_0 ) ( 1 + (T-t)^{\frac{N}{2} -\frac{k}{p-1} -1}). $$
 Finally, we  derive
 $$ | \partial_t \|U\|_{L^k(\Omega)}^k| \leq C( A, \epsilon_0, \eta_0 ) (T-t)^{-1 + \varepsilon}, $$
 with a small $ \varepsilon \in (0,1)$. Thus, we conclude the result for the case $ \frac{k}{p-1} -\frac{N}{2} < 0$.

 - Now, we start to the case where $ \frac{k}{p-1} - \frac{N}{2} > 0$:  Let us consider an arbitrary $K_0' \geq K_0$, and  we take \eqref{decompose-U-L-k} by replacing $K_0$ to $K_0'$. Then, we repeat the process for the case sub-critical $ \frac{k}{p-1} - \frac{N}{2} < 0$ by using Lemma \ref{lemma-estimate-U-x-t-in-S-t}, \eqref{integral-fundamental} and \eqref{integral-P-1-K-0'}, we can write
 \begin{eqnarray*}
 \int_{|x| \leq \frac{K_0'}{4} \sqrt{(T-t)|\ln(T-t)|}} U^k dx &=& (T-t)^{\frac{N}{2} - \frac{k}{p-1}} |\ln(T-t)|^{\frac{N}{2}} \int_0^{\frac{K_0'}{4}} \varphi_0(r) r^{N-1} dr \\
 & + & o_{t \to T}( (T-t)^{\frac{N}{2} - \frac{k}{p-1}} |\ln(T-t)|^{\frac{N}{2}}),
  \end{eqnarray*}
  and
\begin{eqnarray*}
\left| \int_{\frac{K_0'}{4} \sqrt{(T-t)|\ln(T-t)|} \leq |x| \leq \epsilon_0 } U^k\right| & \leq & ( K_0')^{\frac{N}{2} -\frac{k}{p-1}}.C(p). (T-t)^{\frac{k}{p-1} -\frac{N}{2}} |\ln(T-t)|^{\frac{N}{2}} \\
&+& o_{t \to T}( (T-t)^{\frac{N}{2} - \frac{k}{p-1}} |\ln(T-t)|^{\frac{N}{2}}),
\end{eqnarray*}
and
$$ \int_{\{|x| \geq \epsilon_0 \} \cap \Omega} U^k \leq C(\epsilon_0, \eta_0).$$
Let us define 
$$ v(t) = (T-t)^{\frac{N}{2} -\frac{k}{p-1}} |\ln(T-t)|^{\frac{N}{2}}.$$
We see that $K_0'$ is free, then, taking $K_0' \to +\infty$, we derive 
$$\|U\|^k_{L^k(\Omega)} =\left(  \int_0^\infty \varphi_0(r) r^{N-1} dr \right) v(t) + o_{t \to T}(v(t)). $$

- The case $ \frac{k}{p-1} - \frac{N}{2} =0$: This case arises a different situation.  Indeed, we use again 
\begin{eqnarray*}
\partial_t \|U\|^k_{L^k(\Omega)} &=&  k \int_\Omega     U^{k-1} U_t  dx  \\
&=&  k \int_\Omega \Delta U U^{k -1} + k \int_{\Omega}  U^{  p-1 + k  }    + k\left( \frac{1}{p-1} \frac{\theta'(t)}{\theta(t)}  - 1\right)  \int_{\Omega} U^k.
 \end{eqnarray*}
 We now use the result in item $(ii)$ with $k' = p-1+k$, and we derive
 \begin{eqnarray*}
 \int_{\Omega} U^{p-1+k} = \left(   \int_0^\infty \varphi_0^{p-1+k} r^{N-1}  + o_{t \to T}(1) \right) (T-t)^{-1}|\ln(T-t)|^{\frac{N}{2}}.
 \end{eqnarray*}
 In addition to that, repeating the process in the proof of item $(ii)$ by using Lemma \eqref{lemma-estimate-U-x-t-in-S-t}, \eqref{integral-fundamental} and \eqref{integral-P-1-K-0'}, we derive
 \begin{eqnarray*}
 \int_{\Omega} \Delta U U^{k-1} = (1 - k) \int_{\Omega
 } |\nabla U|^2 U^{k-2} = o_{t \to T}\left( (T-t)^{-1}|\ln(T-t)|^{\frac{N}{2}} \right),
 \end{eqnarray*}
 and use more  
 \begin{eqnarray*}
\left( \frac{1}{p-1} \frac{\theta'(t)}{\theta(t)}  - 1\right)  \int_{\Omega} U^k =o_{t \to T}\left( (T-t)^{-1}|\ln(T-t)|^{\frac{N}{2}} \right).
 \end{eqnarray*}
 Then, we derive
 $$ \partial_t \|U\|^k_{L^k(\Omega)} =  \left(   k\int_0^\infty \varphi_0^{p-1+k} r^{N-1}  + o_{t \to T}(1) \right) (T-t)^{-1}|\ln(T-t)|^{\frac{N}{2}}.  $$
 Thus, we derive
 $$ \|U\|^k_{L^k(\Omega)} =  \left(   \frac{k}{\frac{N}{2}+1}\int_0^\infty \varphi_0^{p-1+k} r^{N-1}  + o_{t \to T}(1) \right) |\ln(T-t)|^{\frac{N}{2}+1}.$$
 
\end{proof}

\section{Finite dimensional   reduction}

In the current section, we  try to reduce the infinite dimensional problem of controlling $U(t) \in S(t)$ to the finite dimensional problem of controlling the two positive spectrum modes  $q_0$ and $q_1$ in $\hat{\mathcal{V}}_A (s)$. More precisely, this reduction concludes Step 1 of  Proposition  \ref{c4proposition-existence-U}
\begin{proposition}[Reduction to a  finite  dimensional problem]\label{c4proposition-reduction-finite}  There exist positive  $T,   K_0, \epsilon_0, \alpha_0, A, \delta_0,  C_0$ and $ \eta_0$ such that  the following holds:  Assume that    $(d_0,d_1) \in \mathcal{D}$ and  the solution  $U$ of equation \eqref{equa-U-theta-'-theta}   exists on $[0, t_1],$ for some $t_1 < T,$  with initial data $U_{d_0, d_1}$ as defined  in  \eqref{c4defini-initial-data}.  Furthermore, we assume that we have  $ U \in   S (t)$ for all $ \forall t \in [0, t_1] \text{ and }  U(t_1) \in  \partial S (t_1)$ (see the  definition of $S(t) = S( K_0, \epsilon_0, \alpha_0, A, \delta_0,  C_0, \eta_0,t)$ in Definition \ref{defini-shrinking-set-S-t} and  the  set $\mathcal{D}_A$ given  in  Proposition \ref{c4proposiiton-initial-data}). Then, the following  statements  hold: 
	\begin{itemize}
	\item[$(i)$] We have  $(q_0, q_1) (s_1)  \in \partial \hat{\mathcal{V}}_A (s_1),$ where $q_0, q_1$ are the components of $q(s)$ given by \eqref{c4represent-non-perp} and $s_1=\ln(T-t_1).$
	\item[$(ii)$]   There exists  $\nu_0 > 0$ such that  for all $\nu \in (0, \nu_0),$ we have
	$$  (q_0, q_1) ( s_1  +  \nu) \notin    \hat{\mathcal{V}}_A (s_1 +  \nu).$$ 
	Consequently,  there  exists  $\nu_1 > 0$ such that 
	$$ U  \notin    S(t_1  +  \nu), \forall \nu \in (0, \nu_1).$$
	\end{itemize}
	\end{proposition}    
\begin{proof}
The proof is quite the same as \cite[Proposition 3.4]{DZM3AS19} that mainly   relies  on  \textit{a priori  estimates } in regions $P_1, P_2$ and $P_3$  of $S(t)$ introduced in Definition \ref{defini-shrinking-set-S-t}. For that reason, we ignore the detailed arguments and we only provide below the required a priori estimates to conclude the proof.
\end{proof}

\subsection{A priori estimates on $P_1$} 

We aim to prove the following Lemma:\begin{lemma}\label{c4propo-estima-in-P-1}
There exists   $K_4 > 0, A_4 > 0$ such that    for all $K_0 \geq K_4, A  \geq A_4$ and  $l^* > 0$  there exists   $T_4 (K_0,  A,  l^* )$   such that for all positive $ \epsilon_0, \alpha_0,  \delta_0, \eta_0,  C_0$ and for $T \leq T_4, l  \in [0, l^*]$,   the following holds: \\ Assume that  we   have the following conditions: 
\begin{itemize}
\item[-]  We consider initial data  $ U_{d_0,d_1} $, defined as in \eqref{c4defini-initial-data} and  $(d_0, d_1) \in \mathcal{D}_{A},$ given in Proposition  \ref{c4proposiiton-initial-data}   such that    $(q_0,q_1)(s_0)$ belongs  to  $\hat{\mathcal{V}}_A(s_0)$ (defined  in  \eqref{c4defini-hat-mathcal-V-A}), where  $s_0 = -  \ln T,$  $\hat{\mathcal{V}}_A (s)$ is  and $q_0, q_1$ are components  of  $q_{d_0, d_1} (s_0)$.  
 \item[-] We have  $U (t) \in S(T,K_0, \epsilon_0, \alpha_0, A, \delta_0, C_0,\eta_0, t)$  for all $t  \in [T-e^{-\sigma}, T - e^{-(\sigma + l)}]$, for some  $\sigma \geq s_0$ and $l \in [0, l^*]$.
 \end{itemize}
Then, the following estimates hold:
\begin{itemize}
\item[$(i)$] For all  $s  \in  [ \sigma,  \sigma + l	]$, we have
\begin{eqnarray}
\left|q_0'(s)  -  q_0 (s)\right|  + \left| q'_{1,i} (s)  - \frac{1}{2} q_{1,i} (s) \right| \leq  \frac{C}{s^2}, \forall i \in  \{ 1,...,N \} \label{c4esti-ODE-0-1},
\end{eqnarray}
and
\begin{equation}\label{c4esti-ODE-2}
\left|  q'_{2,i,j}  (s)   + \frac{2}{s} q_{2,i,j} (s)    \right| \leq \frac{C A}{s^3}, \forall i,j \in \{1,...,N\},
\end{equation}
where  $q_1  = (q_{1,j})_{1 \leq i  \leq N }, q_{2}  = (q_{2,i,j})_{1 \leq i,j \leq N}$ and  $q_1,q_2$ are  defined  in \eqref{c4defini-R-i}. 
\item[$(ii)$] Control  of  $q_-(s) $:  For all $s \in [\sigma, \sigma + l] $ and $ y \in \mathbb{R}^N$, we have  the following  two cases:
\begin{itemize}
\item[-] When  $\sigma > s_0$ then:
\begin{eqnarray}
\left|  q_-(y,s)   \right| \leq  C \left(  A e^{-\frac{s -\sigma}{2}}  + A^2 e^{- (s - \sigma)^2}  + (s -\sigma)   \right) \frac{(1 +  |y|^3)}{ s^2}, \label{c4estima-Q--case-sigma-geq-s-0}
\end{eqnarray}
\item[-]  When $\sigma = s_0$ then:
\begin{equation}\label{c4estima-Q--case-sigma=s-0}
\left| q_-(y,s)  \right| \leq  C (1  +  (s -\sigma) ) \frac{(1 + |y|^3)}{s^2}.
\end{equation}
\end{itemize}
\item[$(iii)$] Control  of the  gradient   term of $q$: For all  $s \in [\sigma, \sigma + l] , y \in \mathbb{R}^N$, we have  the two   following  cases:
\begin{itemize}
\item[-]  When  $\sigma > s_0$ then: 
\begin{equation}\label{c4estima-nabla-Q--case-sigma-geq-s-0}
    \left|  (\nabla q)_\perp (y,s)    \right| \leq  C \left(  A e^{-\frac{s -\sigma}{2}}  +  e^{- (s - \sigma)^2}  + (s -\sigma)  + \sqrt{s - \sigma}  \right) \frac{(1 +  |y|^3)}{ s^2}, 
\end{equation}
\item[-]   When $\sigma = s_0$ then:
\begin{equation}\label{c4estima--nabla-Q--case-sigma=s-0}
\left| (\nabla q)_\perp (y,s)  \right| \leq  C \left(1  +  (s -\sigma) + \sqrt{s -\sigma} \right) \frac{(1 + |y|^3)}{s^2}.
\end{equation}
\end{itemize}
\item[$(iii)$]  Control   of the  outside part $q_e$: For all $s \in [\sigma, \sigma + \lambda] ,$ we have  the two   following cases:
\begin{itemize}
\item[-]   When $\sigma >s_0$ then:
\begin{eqnarray}
\left\|  q_e(.,s)   \right\|_{L^\infty(\mathbb{R}^N)}  \leq  C \left(  A^2 e^{-\frac{s -\sigma}{2}}  + A e^{ (s - \sigma)}  + 1 +  (s -\sigma)   \right) \frac{ 1}{ \sqrt{s}}, \label{c4estima-Q-e-case-sigma-geq-s-0}
\end{eqnarray}
\item[-] When $\sigma = s_0$ then:
\begin{equation}\label{c4estima-Q-e-case-sigma=s-0}
\left\| q_e(.,s)  \right\|_{L^\infty(\mathbb{R}^N)} \leq  C \left(1  +  (s -\sigma) \right)\frac{1}{\sqrt s}.
\end{equation}
\end{itemize}
\end{itemize}
\end{lemma} 
\begin{proof}
Let us remark that the proof is similar to \cite[Lemma 3.2 ]{MZnon97} (see also \cite{MZdm97}). Indeed,   our situation corresponds to  equation (24) in   \cite[Lemma 3.2 ]{MZnon97} with $a =0$ and  with  a small  perturbation $G(w,W).$ In particular,    the structure of our shrinking set $S(t)$ in exactly the same as \cite[Lemma 3.2 (i)]{MZnon97}. For those reasons, we kindly  refer the readers  to see  the technical estimates in this work. 
\end{proof}
The preceding Lemma  directly implies the following:
\begin{proposition}[A priori estimates in $P_1 (t)$]\label{c4propo-priori-P-1-later}
There exist $K_5, A_5 \geq 1$ such that for all $K_0 \geq K_5, A \geq A_5, \epsilon_0 > 0, \alpha_0 > 0, \delta_0 \leq \frac{1}{2} \hat{\mathcal{U}}(0), C_0 > 0 , \eta_0 >0 $, there exists $T_5 (K_0, \epsilon_0,\alpha_0,A,\delta_0, C_0,\eta_0)$ such that  for all $T \leq T_5$, the following holds: If $U$ is a non negative  solution  of equation \eqref{equa-U-theta-'-theta} satisfying  $U (t) \in S(T,K_0,\epsilon_0, \alpha_0,A,\delta_0, C_0,\eta_0, t)$ for all $t \in [0,t_5]$ for some $t_5 \in [0,T)$, and  with initial data $ U_{d_0,d_1}$ given in \eqref{c4defini-initial-data} for some $d_0,d_1 \in \mathcal{D}_A$ (cf. Proposition \ref{c4proposiiton-initial-data}), then, for all $s \in [ -\ln T,- \ln (T-t_5) ]$, we have the following:
\begin{eqnarray*}
 \forall i,j \in \{1, \cdots, n\}, \quad |q_{2,i,j}(s)| & \leq & \frac{A^2 \ln s}{2 s^2}, \label{c4conq_1-2} \\
\left\| \frac{q_{,-}(.,s)}{1 + |y|^3}\right\|_{L^\infty(\mathbb{R}^N)} &\leq &  \frac{A}{2 s^{2}}, \left\|  \frac{(\nabla q(.,s))_{\perp}}{ 1 + |y|^3}  \right\|_{L^\infty(\mathbb{R}^N)} \leq  \frac{A}{2 s^2} \\
 \|q_{e}(s)\|_{L^\infty(\mathbb{R}^N)} & \leq &  \frac{A^2}{2 \sqrt s}. \label{c4conq-q-1--and-e}
\end{eqnarray*}
\end{proposition}
\subsection{A priori estimates in $P_2$:}
In order to derive the required a priori estimates within the region $P_2$ we first need the auxiliary result:
\begin{lemma}[A priori estimates in the intermediate region]\label{c4lemma-max-t-x-0}
  There exist $K_ 6, A_6 > 0, $ such that  for all $K_0 \geq K_6, A \geq A_6,  \delta_6 >0$,  there exists      $\alpha_6 (K_0, \delta_6), C_6 (K_0, A) >  0$ such that for all $\alpha_0 \leq  \alpha_6, C_0 > 0 $,   there exists  $\epsilon_6 (\alpha_0, A, \delta_6, C_0)>0$ such that  for all $\epsilon_0 \leq  \epsilon_6 $, there exists $T_6(\epsilon_0, A, \delta_6,C_0)$ and $\eta_6 (\epsilon_0, A, \delta_0, C_0) > 0$ such that for all $T \leq T_6 , \eta_ 0 \leq  \eta_6,   \delta_0 \leq  \frac{1}{2} \left( p-1 + \frac{(p-1)^2}{4p} \frac{K_0^2}{16} \right)^{-\frac{1}{p-1}}$,   the following  holds: if   $U \in S(T, K_0, \epsilon_0, \alpha_0, A, \delta_0, C_0, \eta_0, t)$ for all $t \in [0, t_*],$ for some $t_* \in [0,T)$, then  for all $|x| \in \left[ \frac{K_0}{4} \sqrt{(T -t_*) |\ln (T-t_*)|}, \epsilon_0\right]$, we have:
\begin{itemize}
\item[$(i)$] For all $|\xi| \leq \frac{7}{4}   \alpha_0 \sqrt{|\ln \varrho(x)|}     $ and $\tau \in \left[  \max \left( 0, - \frac{t(x)}{\varrho(x)}\right), \frac{t_*- t(x)}{ \varrho(x)} \right]  $,  the  transformed function   $\mathcal{U} (x,\xi, \tau)$ defined in \eqref{c4rescaled-function-U}   satisfies  the following: 
\begin{eqnarray}
 \left|  \nabla_\xi  \mathcal{U} (x, \xi, \tau)  \right|   & \leq & \frac{2  C_0}{ \sqrt{|\ln \varrho(x)|}}, \label{c4estima-nabka-mathcal-U-lema}\\
 \mathcal{U}(x, \xi, \tau)  & \geq &   \frac{1}{4} \left( p-1   + \frac{(p-1)^2}{4 p} \frac{K_0^2}{16} \right)^{-\frac{1}{p-1}}, \label{c4estima-mathcal-U-leq-K-0-2}\\
   \left|   \mathcal{U} (x, \xi ,  \tau) \right|  & \leq  &  4. \label{c4mathcal-U-bound-lema}
\end{eqnarray}
\item[$(ii)$] For all $|\xi| \leq  2 \alpha_0 \sqrt{|\ln \varrho(x)|}$ and $\tau_0 =  \max\left( 0, - \frac{t(x)}{\varrho(x)}\right)$: we have
$$ \left|  \mathcal{U}( x, \xi, \tau_0)  - \hat{\mathcal{U}} (\tau_0)    \right| \leq \delta_6 \text{ and  }    \left| \nabla_\xi \mathcal{U}(x, \xi, \tau_0)   \right| \leq \frac{C_6 }{  \sqrt{|\ln \varrho(x)|}}. $$ 
\end{itemize}
\end{lemma}
\begin{proof}
For the proof see \cite[Lemma 4.2 ]{DZM3AS19} or \cite[Lemma 2.6]{MZnon97}. 
\end{proof}
Next by following the same reasoning as in the proof of Proposition 4.2 in \cite{DZM3AS19} and taking also into account Lemma \ref{c4lemma-max-t-x-0} we derive the following result:
\begin{proposition}[A priori  estimates in  $P_2 (t)$]\label{c4propo-a-priori-P-2}
There exist  $K_{7}, A_7 >  0$ such that for all $K_0 \geq  K_7,   A \geq A_7$,  there exists  $  \delta_7 \leq \frac{1}{2} \hat{\mathcal{U}}(0) $ and $C_7 (K_0,A)>0$ such that  for all $\delta_0 \leq \delta_7, C_0 \geq C_7$ there exists $\alpha_7 (K_0, \delta_0)>0$ such that for all $ \alpha_0  \leq \alpha_7$,    there exist $\epsilon_7( K_0, \delta_0, C_0) > 0$  such that for all $\epsilon_0 \leq \epsilon_7$,  there exists $T_7(\epsilon_0, A, \delta_0,  C_0)> 0$  such that  for all $T \leq T_7$ the following holds:  If $  U  \in S(T,K_0, \epsilon_0, \alpha_0, A, \delta_0, C_0,t  )$ for all $t \in [0, t_7]$ for some  $t_7 \in [0,  T)$,  then,  for all  $|x| \in \left[ \frac{K_0}{4} \sqrt{(T- t_*) |\ln (T-t_*)|}, \epsilon_0   \right],$        $|\xi| \leq  \alpha_0 \sqrt{|\ln \varrho(x)|} $  and   $  \tau \in \left[ \max  \left( - \frac{t(x)}{\varrho (x)},0\right), \frac{t_7 - t(x)}{\varrho (x)}\right]  $, we have
$$ \left| \mathcal{U}(x, \xi, \tau_*) - \hat{\mathcal{U}} (x, \xi, \tau_*)    \right| \leq  \frac{\delta_0}{2} \text{ and } \left| \nabla\mathcal{U} (x, \xi ,\tau) \right| \leq  \frac{ C_0}{2 \sqrt{| \ln \varrho(x)|}},$$
where  $\varrho (x) = T- t(x)$. 
\end{proposition}

\subsection{ Priori estimates in $P_3$:}  In that region we claim that the following holds:
\begin{proposition}[A priori estimates   in $P_3$]\label{c4propo-priori-P-3}
Let us  consider  positive $K_0, \epsilon_0, \alpha_0, A,$ $C_0, \eta_0$ and $\delta_0 \in [0, \frac{1}{2} \hat{\mathcal{U}}(0)].$ Then, there exists     $T_8( \eta_0)  > 0$  such that  for all    $T \leq T_8$, the following holds: We assume that    $U$ is a non negative  solution of \eqref{equa-U-theta-'-theta} on $[0, t_8] $ for some $t_8 < T	$, and  $U \in S(K_0,\epsilon_0, \alpha_0, A, \delta_0, C_0, \eta_0,t )$ for all $t \in [0,t_8]$ with  initial data  $U(0) = U_{d_0,d_1}$ given  in \eqref{c4defini-initial-data} for $|d_0|, |d_1| \leq 2$.  Then,  for all  $|x| \geq   \frac{\epsilon_0}{4}$   and  $ t \in (0, t_8],$ we get the a priori estimates:
\begin{eqnarray}
\left| U (x,t)  - U(x,0)  \right|  \leq  \frac{\eta_0}{2}, \label{c4U-U-0-control-P-3}\\
\left| \nabla U (x,t)  -  \nabla e^{t \Delta} U(x,0)  \right|   \leq \frac{\eta_
0}{2}.\label{c4nabla U-nabla-semi-intial}
\end{eqnarray}
\end{proposition} 
\begin{proof}
The proof relies on the well-posedness property  of problem \eqref{equa-U-theta-'-theta}. More precisely, we  directly derive by  \eqref{equa-U-theta-'-theta} the following integral equation
\begin{eqnarray}\label{variation-of-parameters formula}
U(t) = e^{t \Delta} U_0 +  \int_0^t e^{(t-s) \Delta}\left[ U^p(s) + \left( \frac{1}{p-1} \frac{\theta'(s)}{\theta(s)} - 1 \right) U(s)\right]ds,
\end{eqnarray}
where $e^{t\Delta} $ stands for the semi-group of the laplacian $\Delta$ associated with  Neuman boundary conditions. Using now the  rough estimates stem from the definition of $S(t)$ then by virtue of \eqref{variation-of-parameters formula} we derive the desired a priori estimates \eqref{c4U-U-0-control-P-3} and \eqref{c4nabla U-nabla-semi-intial}.
\end{proof}

\section{Appendix}

\subsection{Proof of  Proposition  \ref{propo-bar-mu-bounded} }\label{appexidix-propo-theta}  We only consider the sub-critical case
\begin{equation} \label{defi-sub-critical}
1 - \frac{r \gamma}{p-1}  > 0
\end{equation}
together with super-critical one
\begin{equation}\label{defi-sup-critical}
1 - \frac{r \gamma}{p-1}  < 0,
\end{equation}
 whilst the critical case $1 - \frac{r \gamma}{p-1} =0$ is not considered in the current work.  
 
 \medskip
 \noindent
 Using now  Definition \ref{defi-theta-by-U}, we write
\begin{eqnarray*}
\theta(t) =  |\Omega|^{\frac{\gamma}{1- \frac{r \gamma}{p-1}} } \left(  \int_{P_1(t)}
U^r   + \int_{ \left\{ K_0 \sqrt{(T-t) |\ln(T-t)| }  \le |x| \leq \frac{\epsilon_0}{4} \right\}} U^r  + \int_{P_3(t)} 
U^r   \right)^{ -\frac{\gamma}{1- \frac{r \gamma}{p-1}}}.
\end{eqnarray*}

- The proof of item $(i)$: Since the goal is to prove the integrals in the above line to be bounded, the arguments of  \eqref{defi-sub-critical} and \eqref{defi-sup-critical} are the same. For that reason, we only give the proof involving to \eqref{defi-sub-critical}. Next using that \eqref{defi-sub-critical} entails
$$-\frac{\gamma}{1- \frac{r \gamma}{p-1}} < 0, $$
then in conjunction with hypothesis \eqref{condition-integral-U-0} and  via Definition \ref{defini-shrinking-set-S-t}$(iii)$ we deduce 
$$  \theta(t)  \leq  |\Omega|^{\frac{\gamma}{1- \frac{r \gamma}{p-1}} } \left( \int_{P_3} 
U^r   \right)^{ -\frac{\gamma}{1- \frac{r \gamma}{p-1}}}   \leq C( \eta_0, C_2). $$
Using Lemma \ref{lemma-estimate-U-x-t-in-S-t}, item $(i)$,   we have
\begin{eqnarray*}
\int_{P_1} U^r  \leq  C (T-t)^{  \frac{N}{2} - \frac{r}{p-1} } |\ln(T-t)|^{\frac{N}{2}} \leq  \int_{\Omega} U_0^r,
\end{eqnarray*}
provided that $T \leq T'(K_0, C_2)$ and $ \frac{N}{2} - \frac{r}{p-1} >0$.
Then,  we use again item $(ii)$ in the Lemma to derive
\begin{eqnarray*}
\int_{ \{ K_0 \sqrt{(T-t) |\ln(T-t)|}  \le |x| \leq \frac{\epsilon_0}{4} \}} U^r   & \leq  &  C \int_{|x| \leq \epsilon_0 }  \left[ \frac{|x|^2}{|\ln|x||}\right]^{-\frac{1}{p-1}} dx\\
& \leq   & \int_{\Omega} U_0^r, 
\end{eqnarray*}
provided that  $ \epsilon \leq \epsilon'(K_0, C_2) $ and $ \frac{N}{2} - \frac{r}{p-1} >0 $.
Similarly, we also get
$$ \int_{P_3(t)} U^r  \leq C(\eta_0) \left( \int_\Omega U_0^r +1 \right).$$
Thus, we obtain
\begin{eqnarray}
\theta(t) \geq \frac{1}{C(C_2,\eta_0)}, 
\end{eqnarray}
which  concludes  \eqref{bound-bar-theta}.  

\medskip
\noindent
Next, we proceed with the proof of \eqref{bound-bar-theta-'}:   To this end, we express below the  formula of $\theta'(t)$
\begin{eqnarray*}
\theta'(t) = - \frac{\gamma}{1 - \frac{r \gamma}{p-1}}    \left(  \def\avint{\mathop{\,\rlap{--}\!\!\int}\nolimits} \avint_\Omega
U^r\,dx  \right)^{ -\frac{\gamma}{1- \frac{r \gamma}{p-1}} -1} \times \frac{1}{|\Omega|} \int_\Omega r U_t U^{r-1}  dx, 
\end{eqnarray*}
by which it is evident that it is sufficient to estimate the following integral 
$$ \int_\Omega     U^{r-1} U_t  dx.$$
Note that by virtue of equation  \eqref{equa-U-theta-'-theta} we can write 
\begin{eqnarray}
\int_\Omega     U^{r-1} U_t  dx = \int_\Omega \Delta U U^{r-1} + \int_{\Omega}  U^{  p-1 + r  }    + \left( \frac{1}{p-1} \frac{\theta'(t)}{\theta(t)}  - 1\right)  \int_{\Omega} U^r. \label{integral-U-t-U-r-1}
 \end{eqnarray}
 Then, we derive
 \begin{eqnarray}
 \theta'(t)  = -  \frac{r \gamma}{p-1}  \frac{1}{|\Omega| } \left\{\int_\Omega \Delta U U^{r-1} +\int_\Omega U^{p-1+r} -\int_{\Omega} U^r \right\}.\label{express-theta-'}
 \end{eqnarray}

 Next we recall some  necessary material for the proof:   From Lemma \ref{lemma-estimate-U-x-t-in-S-t}, we roughly estimate
 \begin{equation}\label{asymptotic-U}
   \begin{array}{rcl}
        | \nabla^i_x U (x,t)|  & \leq &  C  (T-t)^{-\frac{1}{p-1} - \frac{i}{2}} \text{ for all }      x \in P_1(t), \\[0.2cm]
     | \nabla_x^{i} U(x,t)| & \leq & C( \delta_0, C_0 ) \left[\frac{|x|^2}{|\ln|x||} \right]^{-\frac{1}{p-1} -\frac{i}{2}} \text{ for all } x \in P_2(t),\\[0.2cm]
     | \nabla^i_x U(x,t)| & \leq &   |\nabla^i_x U_0| + \eta_0  \text{ for all } x \in P_3(t),
   \end{array}
 \end{equation}
 for all $ i \in \{0,1 \}.$
 
Besides that, we use the Green formula to derive 
\begin{equation}\label{green-formula}
    \int_\Omega U^{r-1} \Delta U  = (1 -r) \int_\Omega U^{r-2} |\nabla U|^2,
\end{equation}
for all $U \in W^{2,\infty}$ and $ \frac{\partial U}{\partial \nu}   =0.$

In particular, we have the following fundamental  integral for $a < b \ll 1, n > 0, m <0$ and $ m\neq -1$: 
\begin{equation}\label{fundamental}
  \left|  \int_a^b  (-\ln s )^n s^{m} ds  \right| \leq    C(n,m) \left( (-\ln b)^n b^{1+m} + (-\ln a)^n a^{1 +m} \right).
\end{equation}
- Integrals in $P_1$: Using the first estimate in \eqref{asymptotic-U} we  can  obtain
\begin{eqnarray}
\int_{P_1} U^{p-1 +r}  & \leq & C  (T-t)^{-1 + \frac{N}{2} - \frac{ r}{p-1}} |\ln(T-t)|^{\frac{N}{2}},\\
\int_{P_1} U^{r-2} | \nabla U|^2  & \leq & C  (T-t)^{- 1 + \frac{N}{2} - \frac{r}{p-1}  } |\ln(T-t)|^\frac{N}{2}.
\end{eqnarray}
- Integrals in $P_2$:   In fact, we use the second  estimate in \eqref{asymptotic-U} to obtain the following
\begin{eqnarray*}
| U(x,t) |^{p-1 + r}  \leq C ( \delta_0) \left[ \frac{|x|^2}{|\ln|x||} \right]^{ -1  - \frac{r}{p-1}},\\
| U(x,t)|^{r-2}  |\nabla U(x,t)|^2 \leq C(C_0) \left[ \frac{|x|^2}{|\ln|x||} \right]^{ -1  - \frac{r}{p-1}}.
\end{eqnarray*}
Besides
\begin{eqnarray*}
& &\int_{K_0 \sqrt{(T-t)|\ln(T-t)|} \leq |x| \leq \epsilon_0}\left[ \frac{|x|^2}{|\ln|x||} \right]^{ -1  - \frac{r}{p-1}}  \\
&=& \int_{K_0 \sqrt{(T-t)|\ln(T-t)|}}^{\epsilon_0} \left[ \frac{z^2}{| \ln|z||} \right]^{ - 1  - \frac{r}{p-1}} z^{N-1} dz\\[0.2cm]
& = &  \int_{K_0 \sqrt{(T-t)|\ln(T-t)|}}^{\epsilon_0} (-\ln z)^{1 + \frac{r}{p-1}} z^{ - 3 - \frac{2 r}{p-1} + N} dz:=I,
\end{eqnarray*}
and thus by virtue of  \eqref{fundamental} we deduce
$$ I \leq C  \left(\epsilon_0^{-1 +N - \frac{2r }{p-1}} |\ln \epsilon_0|^{1 + \frac{r}{p-1}}  + (T-t)^{- 1 + \frac{N}{2} - \frac{r}{p-1}} |\ln(T-t)|^{ \frac{N}{2}}  \right),$$
which finally entails
\begin{eqnarray*}
& & \int_{P_2} [ U^{p-1+r} + U^{r-2} |\nabla U|^2  ] \\
&\leq &   C( \delta, C_0)  \left(\epsilon_0^{-1 +N - \frac{2r }{p-1}} |\ln \epsilon_0|^{1 + \frac{r}{p-1}}  + (T-t)^{-1 + \frac{N}{2} - \frac{r}{p-1}} |\ln(T-t)|^{ \frac{N}{2}}  \right).
\end{eqnarray*}
Additionally, since in $P_3$ the gradient $\nabla^i U$ is considered as a small perturbation of  initial data we have
\begin{eqnarray*}
\int_{P_3} [ U^{p-1+r} + U^{r-2} |\nabla U|^2  ] 
\leq    C(\Omega, \eta_0). 
\end{eqnarray*}
Finally, choose  $\varepsilon =\frac{1}{2} \left(\frac{N}{2} -\frac{r}{p-1} \right)$, then,  there exists  $T \leq T'( K_0, \delta_0,C_0, \eta_0 )$  and there holds
$$ \int_{\Omega} \left[ U^{p-1+r} + U^{r-2} |\nabla U|^2  \right] \leq (T-t)^{-1 + \varepsilon} .  $$
Using \eqref{bound-bar-theta}, we have 
$$ \int_\Omega  U^r \leq C( \eta_0, C_2). $$
Regarding to \eqref{express-theta-'},  it concludes \eqref{bound-bar-theta-'}. Thus, item $(i)$ follows.  Clearly, $(ii)$ is  derived by item $(i)$. 

\section{Parabolic estimates with Neumann condition}

In this section we provide some useful parabolic estimates by using well known growth conditions of the Neumann heat kernel. 

Let us consider  the following equation
\begin{equation}\label{equa-linear-Neumann}
\left\{ 
\begin{array}{rcl}
\partial_t U  & = & \Delta U \text{ in } \Omega  \times (0,T)\\
\frac{\partial U}{ \partial \nu } &= & 0 \text{ on } \partial \Omega  \times (0,T) \\
U(t=0)  & = & U_0 \text{ on } \Omega.  
\end{array}\right.
\end{equation}
Problem \eqref{equa-linear-Neumann} generates  the associated semi-group $e^{t\Delta}$ and its solution is given by
$$ U(t) =  e^{t \Delta}( U_0) = \int_{\Omega} G(x,y,t) U_0(y) dy, $$
and thus
$$ \frac{\partial e^{t \Delta} U_0 }{\partial \nu} = 0 \text{ for all } t >0 \text{ and } x \in \partial \Omega. $$
In particular,  the Neumann heat kernel $G(x,y,t) $ satisfies the following growth estimates
\begin{equation}\label{proper-Het-Kerm=nel}
    \left| \nabla_x^i G(x,y,t)      \right| \leq C t^{ - \frac{N + i}{2} }\exp \left( - C(\Omega) \frac{|x-y|^2}{t} \right), C(\Omega) >0, \forall i=0,1,
\end{equation}
for all $x,  y \in \Omega$ and  $ t\neq 0,$ cf.  \cite[Lemma 3.3]{YZPA13}.

Now we can prove the following:
\begin{lemma}Let us consider initial data $ U_0(d_0, d_1)$ defined as in   \eqref{c4defini-initial-data} and define
$$ L(t):= e^{t \Delta } (U_0(d_1,d_2)). $$
Then, 
$$ \| L(t)\|_{L^\infty(\Omega \cap \{|x| \geq \frac{\epsilon_0}{8}\})} + \|\nabla L(t)\|_{L^\infty(\Omega \cap \{|x| \geq \frac{\epsilon_0}{8}\})} \leq C(\epsilon_0). $$
\end{lemma}
\begin{proof}
We observe that the first desired estimate for $L$  follows directly from \eqref{proper-Het-Kerm=nel} with $i=0$.  It remains to prove the second one. To this end the technique from \cite[Lemma 4.3]{DZM3AS19} can be applied thanks to estimate \eqref{proper-Het-Kerm=nel}. So, we kindly refer the readers to check the proof of this lemma.      \end{proof}

\begin{lemma}[Parabolic estimates]   Let us consider positive parameters $T,K_0, \epsilon_0, \alpha_0, A, \delta_0, C_0 $ and $\eta_0$ such that Proposition \ref{propo-bar-mu-bounded} holds true and we assume furthermore that 
$$U(t) \in S(K_0, \epsilon_0, \alpha_0, A, \delta_0, C_0,\eta_0,t), \forall t \in [0,t_1).$$ Then we have the following
\begin{eqnarray}
|\nabla^i U(x,t)| \leq C (T-t)^{-\frac{1}{p-1} -\frac{i}{2}},\label{parabolic-estimate-fornabla-i-U}
\end{eqnarray}
for all $i=0,1$ and for all $(x,t) \in \Omega \times  [0,t_1)$.  In particular,  for all $x \in \Omega$, there exists $R_x >0$ such that
\begin{equation}
    |\partial_t U(x,t)| \leq C(A, T,R_x).\label{estimate-partial-t-U}
\end{equation}
\end{lemma}
\begin{proof} Estimate
\eqref{parabolic-estimate-fornabla-i-U} is directly derived  by Definition  \eqref{defini-shrinking-set-S-t}. Besides, the proof of \eqref{estimate-partial-t-U} is quite the same as the proof estimate (F.4) in \cite[Lemma F.1]{DZM3AS19}.

\end{proof}

\subsection{ Some estimates  of terms in  equation \eqref{c4equa-Q}}
In this part, we aim to estimate the terms $ V, B,R$ and $G$ in equation \eqref{c4equa-Q}
\begin{lemma}\label{lemma-V-B-R}
For all $A,K_0 \geq 1$, there exists $s_7(A,K_0) \geq 1$, such that  for all $s \geq s_7  $ and $q \in V_A(s)$, then the following hold: 
\begin{itemize}
    \item[$(i)$]  Estimate for $V$, defined as in \eqref{c4defini-chi-y-s}:
    $$ |V(y,s)| \leq C \text{ and } \left| V(y,s) + \frac{(|y|^2 -2N)}{ 4s}  \right| \leq \frac{C(1+|y|^4)}{s^2}, \forall y\in \R^N.$$
    \item[$(ii)$] Estimate for $B(q)$, defined as in \eqref{c4defini-B-Q}:  
    $$ \left| \chi(y,s) B(q) \right| \leq C|q|^2 \text{ and } |B(q)|  \leq C|q|^{\bar p}, \text{ where } \bar p= \min(p,2),$$
    for all $y\in \R^N.$
    \item[$(iii)$] Estimate for  $R(y,s)$: 
    \begin{eqnarray*}
    \left|R(y,s) \right| &\leq & \frac{C}{s},\\
    \left|  R(y,s) - \frac{c_1}{s^2}\right| &\leq & \frac{C(1+|y|^4)}{s^3},\\
    \left|  \nabla R(y,s) \right| & \leq & \frac{C(|y| +|y|^3)}{\bar p+1},
    \end{eqnarray*}
    for all $y\in \R^N.$
\end{itemize}
\end{lemma}
\begin{proof}
The proof of item $(i)$ arises from a simple Taylor expansion. The proofs of $(ii)$ and $(iii)$ can be found in  of \cite[Lemma 3.15]{MZdm97} and \cite[Lemma B.5]{MZnon97}. 
\end{proof}
Next, Proposition \ref{propo-bar-mu-bounded} ensures that $G$, defined \eqref{c4defini-N-term}, decays exponentially. In particular there holds:
\begin{lemma}[Size of $G$] \label{size-G}
There exists $K_7 \geq 1 $ such that  for all  $K_0 \geq K_{7},  \delta_{0 } > 0,  $ there exists  $\alpha_{7} (K_0, \delta_{0}) > 0$  such that for all  $\alpha_0 \leq   \alpha_{7}, M_0>0$ there exists       $\epsilon_{7} (K_0, \delta_{0}, \alpha_0,M_0) > 0$ such that   for all   $\epsilon_0  \leq \epsilon_7$  and  $A \geq 1, C_0 > 0, \eta_0 > 0$, there exists $T_{7}   > 0$ such that  for  all  $T  \leq T_7$ the following holds:  Assuming   $U$ is a non negative  solution  of   equation \eqref{equa-U-theta-'-theta}    on $[0, t_1],$ for some $t_1 < T$ and  $U \in S(K_0, \epsilon_0, \alpha_0, A, \delta_0, C_0,\eta_0,t)= S(t) $ for all  $t \in [0, t_1]$ with initial data $U_0$ introduced as in \eqref{c4defini-initial-data} for $|d_1|, |d_2| \leq 2$. Then,  $G$, defined as in \eqref{c4defini-N-term} satisfies 
$$ \|G\|_{L^\infty (\R^N)}  \leq e^{-\eta s}, \forall s \in [-\ln T, \ln(T-t_1)],$$
for some $\eta >0$ and small.
\end{lemma}
\begin{proof}
We note that  the constants in the hypothesis mainly satisfy the assumptions of Proposition \ref{propo-bar-mu-bounded}, thus 
$$ |\theta'_t(t)| \leq   C (T-t)^{-1 + \epsilon }, $$
and so we obtain
$$ |\bar \theta'_s (s)| \leq e^{-\epsilon s}.$$
Now, we rewrite $G$ as follows
$$ G = \left(\frac{1}{p-1} \frac{\bar\theta'(s)}{\bar \theta(s)} - e^{s} \right) w + F(w,W),$$
and thus it directly  follows 
$$ \left|   \left(\frac{1}{p-1} \frac{\bar\theta'(s)}{\bar \theta(s)} - e^{s} \right) w\right| \leq C e^{-\min(1,\epsilon) s}.$$
In particular, the nonlinear term $F$ is similar to the one treated in \cite{DZM3AS19} and thus one derives
$$ \|F\|_{L^\infty} \leq C e^{- \epsilon_1 s}. $$
Finally we arrive at the desired estimate for $G$ for $ \eta = \min(1,\epsilon,\epsilon_1)$. \end{proof}

\bibliographystyle{amsalpha}
\bibliography{main}

\def\cprime{$'$}
\providecommand{\bysame}{\leavevmode\hbox to3em{\hrulefill}\thinspace}
\providecommand{\MR}{\relax\ifhmode\unskip\space\fi MR }
\providecommand{\MRhref}[2]{%
  \href{http://www.ams.org/mathscinet-getitem?mr=#1}{#2}
}
\providecommand{\href}[2]{#2}
\begin{thebibliography}{MCHKS18}

\bibitem[BK94]{BKnon94}
J.~Bricmont and A.~Kupiainen, \emph{Universality in blow-up for nonlinear heat
  equations}, Nonlinearity \textbf{7} (1994), no.~2, 539--575. \MR{1267701
  (95h:35030)}

\bibitem[BK19]{BK19}
A.~Bobrowski and M.~Kunze, \emph{Irregular convergence of mild solutions of
  semilinear equations}, Jour. Math. Anal. Applications \textbf{472} (2019),
  1401--1419.

\bibitem[Bre92]{Brejde92}
A.~Bressan, \emph{Stable blow-up patterns}, J. Differential Equations
  \textbf{98} (1992), no.~1, 57--75. \MR{1168971 (93c:35041)}

\bibitem[DKN21]{DKN20}
O.~Drosinou, N.~I. Kavallaris, and C.~V. Nikolopoulos, \emph{A study of a
  non-local problem with robin boundary conditions arising from mems
  technology}, Math. Meth. Appl. Sciences \textbf{to appear} (2021).

\bibitem[DNZ19]{DNZtunisian-2017}
G.~K. {Duong}, V.~T. {Nguyen}, and H.~{Zaag}, \emph{{Construction of a stable
  blowup solution with a prescribed behavior for a non-scaling-invariant
  semilinear heat equation}}, {Tunis. J. Math.} \textbf{1} (2019), no.~1,
  13--45 (English).

\bibitem[DNZ20]{DNZMAMS20}
G.~K. Duong, N.~Nouaili, and H.~Zaag, \emph{Construction of blow-up solutions
  for the complex ginzburg-landau equation with critical parameters}, Mem.
  Amer. Math. Soc. (2020), to appear.

\bibitem[Duo19a]{DJDE2019}
G.~K. Duong, \emph{A blowup solution of a complex semi-linear heat equation
  with an irrational power}, J. Differential Equations \textbf{267} (2019),
  no.~9, 4975--5048. \MR{3991552}

\bibitem[Duo19b]{DJFA2019}
\bysame, \emph{Profile for the imaginary part of a blowup solution for a
  complex-valued semilinear heat equation}, J. Funct. Anal. \textbf{277}
  (2019), no.~5, 1531--1579. \MR{3969198}

\bibitem[DZ19]{DZM3AS19}
G.~K. Duong and H.~Zaag, \emph{Profile of a touch-down solution to a nonlocal
  mems model}, Mathematical Models and Methods in Applied Sciences \textbf{29}
  (2019), no.~07, 1279--1348.

\bibitem[GK85]{GKcpam85}
Y.~Giga and R.~V. Kohn, \emph{Asymptotically self-similar blow-up of semilinear
  heat equations}, Comm. Pure Appl. Math. \textbf{38} (1985), no.~3, 297--319.
  \MR{784476 (86k:35065)}

\bibitem[GK89]{GKcpam89}
\bysame, \emph{Nondegeneracy of blowup for semilinear heat equations}, Comm.
  Pure Appl. Math. \textbf{42} (1989), no.~6, 845--884. \MR{1003437
  (90k:35034)}

\bibitem[GK12]{GKDCDS2012}
J.~S. Guo and N.~I. Kavallaris, \emph{On a nonlocal parabolic problem arising
  in electrostatic {MEMS} control}, Discrete Contin. Dyn. Syst. \textbf{32}
  (2012), no.~5, 1723--1746. \MR{2871333}

\bibitem[GM72]{GMK1972}
A.~Gierer and H.~A. Meinhardt, \emph{Theory of biological pattern formation},
  Kybernetik. \textbf{12} (1972), 30--39.

\bibitem[GNZ16]{GNZsysparabolic2016}
T.~Ghoul, V.~T. Nguyen, and H.~Zaag, \emph{Construction and stability of blowup
  solutions for a non-variational semilinear parabolic system}, submitted:
  (2016). \MR{3484968}

\bibitem[GS15]{GSSAIM2015}
J.S. Guo and P.~Souplet, \emph{No touchdown at zero points of the permittivity
  profile for the {MEMS} problem}, SIAM J. Math. Anal. \textbf{47} (2015),
  no.~1, 614--625. \MR{3305368}

\bibitem[KBM21]{KBM20}
N.~I. Kavallaris, R.~Barreira, and A.~Madzvamuse, \emph{Dynamics of shadow
  system of a singular gierer-meinhardt system on an evolving domain}, Jour.
  Nonl. Science \textbf{31} (2021), no.~5.

\bibitem[Kee78]{KSAM78}
J.~P. Keener, \emph{Activators and inhibitors in pattern formation}, Stud.
  Appl. Math. \textbf{59} (1978), no.~1, 1--23. \MR{0479051}

\bibitem[KL20]{KL20}
N.~I. Kavallaris and E.~A. Latos, \emph{Diffusion-driven blow-up for a
  non-local fisher-kpp type model}, Arxiv (2020).

\bibitem[KLN16]{NIKN16}
N.~I. Kavallaris, A.~A. Lacey, and C.~V. Nikolopoulos, \emph{On the quenching
  of a nonlocal parabolic problem arising in electrostatic mems control}, Nonl.
  Analysis (2016), no.~138, 189--206.

\bibitem[KS17]{KSN17}
N.~I. Kavallaris and T.~Suzuki, \emph{On the dynamics of a non-local parabolic
  equation arising from the gierer-meinhardt system}, Nonlinearity (2017),
  no.~30, 1734--1761.

\bibitem[KS18]{NKMI2018}
\bysame, \emph{Non-local partial differential equations for engineering and
  biology}, Mathematics for Industry (Tokyo), vol.~31, Springer, Cham, 2018,
  Mathematical modeling and analysis. \MR{3752137}

\bibitem[MCHKS18]{MCHKSN18}
A~Marciniak-Czochra, S~H\"{a}rting, G~Karch, and K~Suzuki, \emph{Dynamical
  spike solutions in a nonlocal model of pattern formation}, Nonlinearity
  \textbf{31} (2018), no.~5, 1757--1781. \MR{3816653}

\bibitem[MCM17]{MCMVJM17}
A~Marciniak-Czochra and A~Mikeli\'{c}, \emph{Shadow limit using renormalization
  group method and center manifold method}, Vietnam J. Math. \textbf{45}
  (2017), no.~1-2, 103--125. \MR{3600418}

\bibitem[Mer92]{Mercpam92}
F.~Merle, \emph{Solution of a nonlinear heat equation with arbitrarily given
  blow-up points}, Comm. Pure Appl. Math. \textbf{45} (1992), no.~3, 263--300.
  \MR{1151268 (92k:35160)}

\bibitem[MNZ16]{MNZNon2016}
F.~Mahmoudi, N.~Nouaili, and H~Zaag, \emph{Construction of a stable periodic
  solution to a semilinear heat equation with a prescribed profile}, Nonlinear
  Anal. \textbf{131} (2016), 300--324. \MR{3427983}

\bibitem[MZ97a]{MZnon97}
F.~Merle and H.~Zaag, \emph{Reconnection of vortex with the boundary and finite
  time quenching}, Nonlinearity \textbf{10} (1997), no.~6, 1497--1550.
  \MR{1483553 (99b:35204)}

\bibitem[MZ97b]{MZdm97}
F.~Merle and H.~Zaag, \emph{Stability of the blow-up profile for equations of
  the type {$u_t=\Delta u+\vert u\vert ^{p-1}u$}}, Duke Math. J. \textbf{86}
  (1997), no.~1, 143--195. \MR{1427848 (98d:35098)}

\bibitem[MZ08]{MZjfa08}
N.~Masmoudi and H.~Zaag, \emph{Blow-up profile for the complex
  {G}inzburg-{L}andau equation}, J. Funct. Anal. \textbf{255} (2008), no.~7,
  1613--1666. \MR{2442077 (2010a:35246)}

\bibitem[NZ15]{NZcpde15}
N.~Nouaili and H.~Zaag, \emph{Profile for a simultaneously blowing up solution
  to a complex valued semilinear heat equation}, Comm. Partial Differential
  Equations \textbf{40} (2015), no.~7, 1197--1217. \MR{3341202}

\bibitem[NZ16]{NZsns16}
V.~T. Nguyen and H.~Zaag, \emph{Construction of a stable blow-up solution for a
  class of strongly perturbed semilinear heat equations}, Ann. Scuola Norm.
  Sup. Pisa Cl. Sci., to appear (2016).

\bibitem[NZar]{NZ2017}
N.~Nouaili and H.~Zaag, \emph{Construction of a blow-up solution for the
  complex ginzburg-landau equation in some critical case}, Arch. Rat. Mech.
  Anal (2018. to appear).

\bibitem[Tur52]{TPTRS1952}
A.~M. Turing, \emph{The chemical basis of morphogenesis}, Phil. Trans. R. Soc.
  \textbf{237} (1952), no.~641, 37--72.

\bibitem[TZ19]{TZpre15}
S.~Tayachi and H.~Zaag, \emph{Existence of a stable blow-up profile for the
  nonlinear heat equation with a critical power nonlinear gradient term},
  Trans. Amer. Math. Soc. to appear (2019).

\bibitem[YZ13]{YZPA13}
X.~Yang and T.~Zhang, \emph{Estimates of heat kernels with {N}eumann boundary
  conditions}, Potential Anal. \textbf{38} (2013), no.~2, 549--572.
  \MR{3015364}

\end{thebibliography}

\end{document}